\documentclass[11pt,reqno]{amsart}

\usepackage{amsfonts}
\usepackage{eurosym}
\usepackage{amssymb}
\usepackage{amsthm}
\usepackage{amsmath}
\usepackage{amsaddr}
\usepackage{bm}
\usepackage{cite}
\usepackage{mathrsfs}
\usepackage{xcolor}
\usepackage[OT1]{fontenc}
\usepackage[left=2.5cm, right=2.5cm, top=3cm, bottom=2.5cm]{geometry}
\usepackage{hyperref}
\usepackage{graphicx}
\usepackage{mathtools}
\usepackage{bigints}

\hypersetup{colorlinks=true, linkcolor=blue, citecolor=red, urlcolor=blue}

\usepackage{times}

\usepackage{xpatch}
\makeatletter
\AtBeginDocument{\xpatchcmd{\@thm}{\thm@headpunct{.}}{\thm@headpunct{}}{}{}}
\makeatother

\allowdisplaybreaks
\newtheorem{theorem}{Theorem}[section]

\newtheorem{lemma}[theorem]{Lemma}
\newtheorem{proposition}[theorem]{Proposition}
\theoremstyle{remark}
\newtheorem{remark}[theorem]{Remark}

\newcommand \norm[1] {\left\Vert #1\right\Vert}
\def\n{\textbf{\textit{n}}}
\def\R{\mathbb{R}}

\def\G{\mathcal{G}}
\def\F2o{\overline{F_2}}

\def\d{{\rm d}}
\def \l {\langle}
\def \r {\rangle}

\def\ddt{\frac{\d}{\d t}}

\def\vphi{\varphi}

\def\div{\mathrm{div}\,}

\def \au {\rm}
\def \ti {\it}
\def \jou {\rm}
\def \bk {\it}
\def \no#1#2#3 {{\bf #1} (#3), #2.}
\def \eds#1#2#3 {#1, #2, #3.}
\def \nome#1#2 {{\bf #1}, (#2).}

\newcommand{\sref}[2]{\hyperref[#2]{#1 \ref*{#2}}}

\numberwithin{equation}{section}

\def\d{{\rm d}}

\newcommand{\habil}[1]{}

\newcommand{\uloc}{\operatorname{uloc}}

\newcommand{\eps}{\ensuremath{\varepsilon}}

\def \au {\rm}
\def \ti {\it}
\def \jou {\rm}
\def \bk {\it}
\def \no#1#2#3 {{\bf #1} (#3), #2.}
\def \eds#1#2#3 {#1, #2, #3.}

\makeatletter
\def\@settitle{\begin{center}%
  \baselineskip14\p@\relax
    \huge
  \@title
  \end{center}%
}
\makeatother

\begin{document}

\title[Cahn-Hilliard equation with non-degenerate mobility]
{\emph{New results for the Cahn-Hilliard equation \\ with non-degenerate mobility:\\
well-posedness and longtime behavior}}
\author[Monica Conti, Pietro Galimberti, Stefania Gatti  \& Andrea Giorgini]{Monica Conti, Pietro Galimberti, Stefania Gatti \& Andrea Giorgini}

\address{Politecnico di Milano\\
Dipartimento di Matematica\\
Via E. Bonardi 9, I-20133 Milano, Italy \\
\href{mailto:monica.conti@polimi.it}{monica.conti@polimi.it},
\href{mailto:andrea.giorgini@polimi.it}{andrea.giorgini@polimi.it}}

\address{Università degli Studi di Modena e Reggio Emilia\\ 
Dipartimento di Scienze Fisiche, Informatiche e Matematiche\\ 
Via Campi 213/B, I-41125 Modena, Italy
\\
\href{mailto:stefania.gatti@unimore.it}{stefania.gatti@unimore.it},
\href{mailto:pietro.galimberti@edu.unife.it}{pietro.galimberti@edu.unife.it}
}



\begin{abstract}
We study the Cahn-Hilliard equation with non-degenerate concentration-dependent mobility and logarithmic potential in two dimensions. We show that any weak solution is unique, exhibits propagation of uniform-in-time regularity, and stabilizes towards an equilibrium state of the Ginzburg-Landau free energy for large times. These results improve the state of the art dating back to a work by Barrett and Blowey. Our analysis relies on the combination of enhanced energy estimates, elliptic regularity theory and tools in critical Sobolev spaces. 
\end{abstract}

\maketitle

\section{Introduction and main results}
\label{Sec-Intro}

We consider the Cahn-Hilliard equation
\begin{subequations}
\begin{alignat}{2}
\label{CH1} 
\partial_t \vphi &= \div( b(\vphi) \nabla \mu) \quad &\text{in }  \Omega \times (0,\infty),\\
\label{CH2}
\mu &= -\Delta \vphi + \Psi^{\prime}(\vphi)  \quad &\text{in } \Omega \times (0,\infty),
\end{alignat}
\end{subequations}
 equipped with the following boundary and initial conditions 
\begin{equation}  \label{nCH-mu}
\partial_\n \vphi= b(\varphi) \partial_\n \mu=0 \quad \text{on } \, \partial \Omega \times (0,\infty),\quad
\vphi|_{t=0}= \vphi_0 \quad \text{in } \, \Omega.
\end{equation}
Here, $\Omega$ is a bounded domain in  $\R^d$, $d \leq 3$, and $\n$ is the outward normal vector on $\partial \Omega$. The state variable $\varphi$ represents the difference of phase concentrations in a binary mixture, and takes values in the interval $[-1,1]$, with $\vphi= \pm 1$ being the pure phases. 
The function $b: [-1,1] \to \mathbb{R}$ is the so-called Onsager mobility, which measures the strength of the diffusion. The mobility can assume one the following forms
\begin{align}
(1) \qquad &b(\cdot) \equiv m_0 \in \mathbb{R}_+ &&\text{(constant case)}, \label{m-con}\\
(2) \qquad &b \in C([-1,1]): 0 < b_m \le b(s) \le b_M, \ s \in [-1,1] &&\text{(non-degenerate case)},
\label{m-ndeg}\\
(3) \qquad &b(s)= m_0(1-s^2), \ s \in [-1,1], \ m_0\in \mathbb{R}_+  &&\text{(degenerate case)}.
\label{m-deg}
\end{align}
The double-well function $\Psi$ is the Flory-Huggins potential (also called Boltzmann-Gibbs entropy) 
\begin{equation}\label{f-log} 
\Psi(s)=F(s)- \frac{\Theta_0}{2}s^2 = \frac{\Theta}{2}\Big[ (1+s)\ln(1+s)+(1-s)\ln (1-s)\Big] - \frac{\Theta_0}{2}s^2,\quad s\in [-1,1], 
\end{equation}
where the positive parameters $\Theta$ and $\Theta_0$ satisfy the condition $\Theta_0 - \Theta>0$. 
The function $\mu$ is the so-called chemical potential, that is the variational derivative of the Ginzburg-Landau free energy
\begin{equation*}
E(\vphi) = \int_{\Omega}\frac{1}{2} |\nabla \vphi|^2 + \Psi(\vphi) \, \d x.
\end{equation*} 

The Cahn-Hilliard equation was originally proposed as a simple model for the process of phase separation in binary alloys at fixed temperature \cite{CH, CH71}. More recently, the Cahn-Hilliard model has become a paradigm for liquid-liquid phase separation arising in polymer/solvent mixtures and 
in living tissues, such as tumor cells \cite{SB}. The macroscopic motion of particles is driven by the minimization of the free energy $E(\varphi)$ through a gradient flow structure, where the chosen metric may depend on the  concentration-dependent mobility, namely different strengths of the mobility function tune the minimization of the energy.

In the context of the mathematical analysis of the Cahn-Hilliard equation, a robust theoretical framework has been largely developed in the constant mobility case \eqref{m-con}. Specifically, the global well-posedness of weak solutions and the propagation of regularity have been extensively studied in \cite{ABELS2009, AW2007, DD, EL1991, KNP, GGW2018}. The analysis of the so-called separation property has been developed in \cite{GGG2023, GP2024, GGM, HW2021, MZ}, and the characterization of the longtime behavior has been addressed in \cite{AW2007, MZ}.

In contrast, much less is known when the mobility function depends on the concentration $\varphi$. This is due to the highly nonlinear character of the equation \eqref{CH1} when $b$ is non-constant. On the other hand, a theory for weak solutions is still expected by the presence of the (formal) {\it energy balance}: 
\begin{equation}
E(\vphi(t))+ \int_0^t \int_{\Omega} b(\vphi) |\nabla \mu|^2 \, \d x \, \d s = E(\vphi_0), \quad \forall \, t \geq 0.
\end{equation}
In the degenerate mobility case \eqref{m-deg}, the only available result concerns with the global existence of weak solutions. This was shown in the seminal work \cite{EG1996}, see also \cite{YIN} for the one-dimensional case, and \cite{DD2016} for the case with regular potentials. More recently, a gradient flow approach has been successfully explored in \cite{CMN2019} and \cite{LMS2012}. We also mention the work \cite{SZ2013} analyzing a different class of degenerate mobility and singular potential functions. In the non-degenerate case \eqref{m-ndeg}, the global existence of weak solutions was established in \cite{BB1999}, following the approach developed in \cite{EG1996}. Besides, the authors of \cite{BB1999} proved the local existence and uniqueness of strong solutions in three dimensions. Furthermore, in the two dimensional case, the existence of strong solutions was claimed to hold globally in time. However, a mistake in the proof undermines the validity of this result, as we will discuss below. 
Concerning the longtime behavior, the existence of a global attractor was demonstrated in \cite{S2007} for generalized semiflows lacking the uniqueness of solutions. Lastly, we mention that the existence of weak solutions for mobility functions depending on the chemical potential, as proposed in \cite{Gurtin}, was proven in \cite{GMRS2011}. 

In this work, we will focus on the Cahn-Hilliard equation with non-degenerate concentration-dependent mobility. For the reader's convenience, we detail the results proven by Barrett and Blowey in \cite{BB1999}, which represent the state of the art about the well-posedness theory. We refer to Section \ref{S-Mathset} for the notation.
\begin{theorem}[Barrett \& Blowey, 1999]
\label{BB} 
Assume that $\Omega \subset \mathbb{R}^d$, with $d=2,3$, is a convex polyhedron or $\partial \Omega \in C^{1,1}$, and the mobility $b$ satisfies \eqref{m-ndeg}. Then, the following holds:

\begin{itemize}
\item[(1)] Let $\vphi_0 \in H^1(\Omega) \cap L^{\infty}(\Omega)$ be such that $\norm{\vphi_0}_{L^{\infty}(\Omega)}\le 1$ and $\overline{\vphi_0}= \frac{1}{|\Omega|} \int_\Omega \varphi_0 (x) \, \d x \in (-1,1)$. Then, there exists a global 
weak solution $\left( \vphi, \mu \right)$  to problem \eqref{CH1}-\eqref{CH2} and \eqref{nCH-mu} satisfying
\begin{align}
\label{BBweak1}
&\vphi \in L^{\infty}(0,T; H^1(\Omega)) \cap L^2(0,T; H^2(\Omega)), \quad 
\partial_t \varphi \in L^2(0,T; H_{(0)}^{-1}(\Omega) ),\\
\label{BBweak2}
&\mu \in L^2(0,T; H^1(\Omega)), \quad \Psi'(\vphi) \in L^2(0,T; L^2(\Omega)),
\end{align}
for all $T\geq 0$,  such that
\begin{alignat}{2}
\label{w-BB1}
&\l\partial_t \vphi, v \r + ( b(\vphi) \nabla \mu, \nabla v) = 0, \quad &&\forall \, v \in H^1(\Omega), 
\\
\label{w-BB2}
&(\mu, v) = ( \nabla \vphi, \nabla v) + ( \Psi'(\vphi), v), \quad &&\forall \, v \in H^1(\Omega),
\end{alignat}
for almost every $t \in (0,\infty)$.
\smallskip

\item[(2)] Suppose, in addition, that $b \in C^1([-1,1])$ and
\begin{equation}
\label{uniq-cond}
\quad \partial_t \varphi \in L^{\frac{8}{8-d}}(0,T; L^2(\Omega)),
\quad 
\varphi \in L^{2^d}(0,T; H^2(\Omega)), 
\quad
\mu \in L^{\frac{8}{6-d}}(0,T; H^1(\Omega)).
\end{equation}
Then, the solution $(\varphi,\mu)$ to \eqref{w-BB1}-\eqref{w-BB2} is unique on $[0,T]$.
\smallskip

\item[(3)] Let $\varphi_0 \in H^3(\Omega)$ be such that $\norm{ \varphi_0}_{L^{\infty}(\Omega)} < 1$, and $\partial_{\n} \vphi_0 = 0$ on $\partial \Omega$. 
Assume that $b \in C^1([-1,1])$. 
Then, the weak solution $(\varphi, \mu)$ to \eqref{w-BB1}-\eqref{w-BB2} satisfies
\begin{align}
&\vphi \in L^{\infty}(0,T_{\ast}; H^2(\Omega)) \cap H^1(0,T_{\ast};  H^1(\Omega)), \quad 
\partial_t \varphi \in L^\infty(0,T_{\ast};  H_{(0)}^{-1}(\Omega)), \\
&\mu \in L^{\infty}(0,T_{\ast}; H^1(\Omega)) \cap 
L^2(0,T_{\ast}; H^2(\Omega)),
\end{align}
where $T_{\ast}=T$ for all $T>0$ if $d = 2$, and $T_{\ast}>0$ depends on the norms of the initial datum if $d=3$.
\end{itemize}
\end{theorem}

The results presented in Theorem \ref{BB} exhibit some shortcomings. Firstly, the condition in \eqref{uniq-cond} on the time derivative $\partial_t \varphi$ is never reached by weak solutions (cf. \eqref{BBweak1}-\eqref{BBweak2}). Thus, no uniqueness result regarding the weak solutions is known, leaving this issue an open question in the Cahn-Hilliard theory. Secondly, concerning the global existence of strong solutions in the two dimensional case (cf. Theorem \ref{BB} - (3)), we point out that the proof of \cite[Corollary 2.1]{BB1999} contains a slight but crucial error. Precisely, in the so-called {\it second energy estimates} involving 
$$
\Lambda(t)= \norm{\sqrt{b(\varphi(t))} \nabla \mu(t) }_{L^2(\Omega)}^2\!,
$$
the authors obtained the following differential inequality  
\begin{equation}
\label{BB-ineq}
\ddt  \Lambda (t)
\leq C \left( 1+\Lambda^{2(1+\rho)}(t) \right)\!, \quad \forall \, \rho \in \left( 0 , \frac12 \right) \!,
\end{equation}
where the positive constant $C$ is independent of $\rho$ (cf. \cite[Eqn. (2.51)]{BB1999}). Relying on the arbitrary smallness of $\rho$, the authors inferred a global control on $\Lambda(t)$. However, as we will show below, a thorough inspection of their proof reveals that the correct constant should be of the form $\frac{C}{\rho}$ (cf. \eqref{final-1} below). This negatively affects the conclusion of the ODE argument used in \cite{BB1999} (cf. \eqref{concl-1}-\eqref{concl-11}). 

The goal of this contribution is to address the gaps left open by Theorem \ref{BB} concerning the well-posedness theory of weak solutions to the  
Cahn-Hilliard equation with non-degenerate mobility in two dimensions. First, we will prove the uniqueness of weak solutions. Then, we will show that the weak solutions exhibit a global smoothness effect at positive times (propagation of regularity). As a consequence, we will prove that any weak solution converges towards a single steady state as $t\to \infty$. 

In order to rigorously state our achievements, we direct the reader to Section \ref{S-Mathset} for the necessary notation, in particular concerning the Bochner spaces $BC([0,\infty);X)$ and $L_{\rm uloc}^p([0,\infty); X)$ of functions that are respectively bounded and continuous in time, and locally uniformly integrable in time, with values in a Banach space $X$. 
\medskip

Our main result is as follows: 

\begin{theorem}\label{Goal_thm}
Let $\Omega \subset \mathbb{R}^2$ be a bounded domain with  $\partial \Omega$ of class $C^3$. Assume that the mobility $b$ satisfies \eqref{m-ndeg}.
Then, the following results hold:
\begin{itemize}
\item[(A)] \textbf{Existence of weak solutions.} Let $\vphi_0 \in H^1(\Omega)\cap L^\infty(\Omega)$ be such that $\norm{\vphi_0}_{L^{\infty}(\Omega)} \le 1$ and $\overline{\vphi_0}=m \in (-1,1)$. Then, there exists a weak solution $\vphi: [0,\infty) \times \Omega \to \mathbb{R}$ to \eqref{CH1}-\eqref{CH2} and \eqref{nCH-mu} such that
\begin{align}
\label{rw1}
&\vphi \in BC([0,\infty); H_{(m)}^1(\Omega)) \cap L^4_{\rm uloc}([0, \infty); H^2(\Omega)) \cap L^2_{\rm uloc}([0, \infty); W^{2,p}(\Omega)), 
\\
\label{rw2}
&\vphi \in L^{\infty}(\Omega \times (0, \infty)): \ |\vphi(x,t)| < 1 \quad \text{a.e. in }\ \Omega \times (0, \infty),
\\
\label{rw3}
&\partial_t \vphi \in L^2(0, \infty; H_{(0)}^{-1}(\Omega)), \quad F'(\varphi)\in L_{\rm uloc}^2 ([0,\infty); L^p(\Omega)), 
\\
\label{rw4}
&\mu \in L_{\rm uloc}^2([0,\infty); H^1(\Omega)),
\end{align} 
for any $ 2\leq p < \infty$, and 
\begin{subequations}
\begin{alignat}{2}
\label{wCH1}
&\l\partial_t \vphi, v \r + ( b(\vphi) \nabla \mu, \nabla v) = 0, \quad &&\forall \, v \in H^1(\Omega), \ \text{for a.e.} \, t \in (0,\infty),
\\
\label{wCH2}
&\mu= - \Delta \vphi + \Psi'(\vphi), \quad &&\text{a.e. in } \, \Omega \times (0,\infty),
\end{alignat}
\end{subequations} 
as well as $\partial_\n \varphi=0$ on $\partial \Omega$ almost everywhere on $(0,\infty)$, and $\varphi|_{t=0}=\varphi_0$.
Moreover, the following energy equality holds
\begin{equation}
\label{EE}
E(\vphi(t))+ \int_0^t \int_{\Omega} b(\vphi) |\nabla \mu|^2 \, \d x \, \d s = E(\vphi_0), \quad \forall \, t \geq 0.
\end{equation}

\noindent
\item[(B)] \textbf{Uniqueness of weak solutions.}
Suppose that $b \in C^2([-1,1])$. Let $\varphi_1, \varphi_2$ be two weak solutions originating from the initial conditions $\varphi_1^0, \varphi_2^0$, respectively, such that $\overline{\varphi_1^0}=\overline{\varphi_2^0}$. Then, for any $T>0$, there exists a positive constant $C$ such that
\begin{equation}
\label{UNIQ}
\norm{\vphi_1(t)-\vphi_2(t)}_{H_{(0)}^{-1}(\Omega)} 
\leq C \norm{ \vphi_1^0-\vphi_2^0 }_{H_{(0)}^{-1}(\Omega)}\!, \quad \forall \, t \in [0,T].
\end{equation}
The constant $C$ depends only on the parameters of the system, the final time $T$, and the initial free energies $E(\vphi_1^0)$ and $E(\vphi_2^0)$. In particular, the weak solution is unique.
\smallskip

\noindent
\item[(C)] \textbf{Propagation of regularity.} Assume that $b \in C^2([-1,1])$. For any $\tau \in (0,1)$, there holds
\begin{align}
\label{REG-phi1}
&\varphi \in L^\infty( \tau,\infty; W^{2,p}(\Omega)), \quad 
\partial_t \varphi \in L^\infty( \tau,\infty; H^1_{(0)}(\Omega)) \cap
L_{\uloc}^2( [\tau,\infty);H^1(\Omega)),\\
\label{REG-phi2}
&\mu \in L^{\infty}(\tau,\infty; H^1(\Omega))\cap L_{\uloc}^2([\tau,\infty);H^3(\Omega)),
\quad F'(\varphi) \in L^\infty(\tau,\infty; L^p(\Omega)),
\end{align}
for any $2 \leq p <\infty$. The equations \eqref{CH1}-\eqref{CH2} are satisfied almost everywhere in $\Omega \times (\tau,\infty)$ and  the boundary condition $\partial_\n \mu=0$ on $\partial \Omega$ holds almost everywhere on $(\tau,\infty)$. In addition, there exists $\delta>0$ depending on $\tau$ and the norms of the initial datum such that 
\begin{equation}
\label{separation}
|\vphi(x,t)|\leq 1-\delta, \quad \forall \, (x,t) \in \overline{\Omega}\times [\tau,\infty).
\end{equation}

\noindent
\item[(D)] \textbf{Convergence to a stationary state.} If $b \in C^2([-1,1])$, the solution $\varphi(t)$ converges to $\varphi_\infty$ in $W^{2-\eps, p}(\Omega)$ as $t \rightarrow \infty$, for any $\eps>0$ and any $2 \leq p <\infty$, where $\varphi_\infty \in W^{2,p}(\Omega)$ is a solution to the stationary Cahn-Hilliard equation
\begin{subequations}
\begin{alignat}{2}
\label{Stat-CH1}
 -\Delta \varphi_\infty + \Psi'(\varphi_\infty)&=\overline{\Psi'(\varphi_\infty)} \quad &&\text{ in } \Omega,
 \\
 \label{Stat-CH2}
 \partial_{\n} \varphi_\infty&=0 \quad  &&\text{ on } \partial\Omega,
\end{alignat}
\end{subequations}
such that  
$$
\frac{1}{|\Omega|} \int_\Omega \varphi_\infty(x) \, \d x = \overline{\varphi_0}.
$$
\end{itemize}
\end{theorem}

\begin{remark} Some remarks are in order:
\begin{itemize}

\item Our analysis can be straightforwardly generalized for {\it singular} potentials $\Psi$ of the form
\begin{equation}
\Psi(s) = F(s) - \frac{\Theta_0}{2} s^2, \quad \forall s \in [-1,1],
\end{equation}
where $F \in C([-1,1]) \cap C^2(-1,1)$  such that $F(0)=F'(0)=0$, 
$\lim_{s \to \pm 1} F'(s) = \pm \infty$, and $F''(s) \ge \Theta >0$, for all s in $(-1,1)$. Furthermore, growth assumptions on $F'$ or $F''$ are required for the separation property \eqref{separation} as in \cite{GGG2023} and \cite{GP2024}.

\item We observe that the validity of the energy equality \eqref{EE} for weak solutions was already shown in \cite[Lemma 2.4]{S2007}.

\item If the initial datum $\vphi_0\in H^2(\Omega)$ is such that $- \Delta \vphi_0 + \Psi'(\vphi_0) \in H^1(\Omega)$, $\overline{\varphi_0} \in (-1,1)$ and $\partial_\n \varphi_0=0$ on $\partial \Omega$, then the corresponding weak solution $\varphi$ is a strong solution, namely $\varphi$ satisfies \eqref{REG-phi1}-\eqref{REG-phi2} and \eqref{separation} with $\tau=0$ (cf. Proposition \ref{CHstrong-solution}).

\item Theorem \ref{Goal_thm} focuses on uniqueness, regularity and longtime behavior of weak solutions. Thus, we require the mobility function $b \in C^2([-1,1])$, which is needed in the uniqueness part. On the other hand, the existence and uniqueness of strong solutions hold under the more general assumption $b \in C^1([-1,1])$ as in Proposition \ref{CHstrong-solution}.
\end{itemize}
\end{remark}

Let us now comment on the main novelties about the proof of Theorem \ref{Goal_thm}. Concerning the uniqueness part, the classical proxy to measure the distance between two weak solutions (say $\varphi_1$ and $\varphi_2$) for the Cahn-Hilliard equation with logarithmic potential is the norm in $H_{(0)}^{-1}$ (cf. definition in \eqref{Spazi-mn}). As already observed in \cite{BB1999}, in the case of concentration-dependent mobility, the appropriate choice is 
\begin{equation}
\label{proxy}
\left\| \sqrt{b(\varphi_1)} \nabla \G_{\varphi_1} (\varphi_1-\varphi_2) \right\|_{L^2(\Omega)},
\end{equation}
where the operator $\G_{\varphi_1}$ is defined in \eqref{def G_q}. In order to derive a suitable differential inequality for this quantity, we will address two crucial steps. Firstly, we need to rigorously justify the integration by parts formula \eqref{IP-time} involving the time derivative of \eqref{proxy}. 
Secondly, we will provide a novel control of the non-linear terms in \eqref{I1}-\eqref{I3}. 
Both steps rely on the combination of two analytical tools: an elliptic regularity theory for a Neumann problem in divergence form with concentration-dependent coefficient (cf. \eqref{epgq}), and the regularity $L_{\uloc}^4([0,\infty); H^2(\Omega))$ for weak solutions, firstly discovered in \cite{GGW2018}. Regarding the propagation of regularity, we will first revise the approach in \cite{BB1999}. By carefully examining the dependence of the constants appearing in the proof, our analysis shows that the correct differential inequality for $\Lambda$ reads as follows:
$$ 
\ddt \Lambda(t)
\leq \frac{C}{\rho} \left( 1+ \Lambda(t) \right)^{2(1+\rho)}, \quad \forall \, \rho \in \left( 0 , \frac14 \right) \!,
$$
for some positive constant $C$ independent of $\rho$, which differs from \eqref{BB-ineq}. Due to the factor $\frac{1}{\rho}$, the argument in \cite{BB1999} do not lead to a global bound for arbitrarily large initial datum. To overcome this issue, we will exploit a recent Grownall-type argument devised in \cite[Lemma 2.1]{liuzhang2023} (see Lemma \ref{gronwallino} below). However, a direct application of this result leads to a global bound which grows double exponentially in time \eqref{Glob_muH1_1}. Then, we propose a new approach to estimate the nonlinear term in \eqref{first est}. Once again, it relies on a careful use of the elliptic regularity theory for \eqref{epgq}. Thanks to the regularity $\varphi \in L_{\uloc}^4([0,\infty); H^2(\Omega))$, the final differential inequality \eqref{final-3} fits into the classical framework of the uniform Gronwall lemma, thus implying the desired global estimate independent of time. Finally, the proof of the convergence of the weak solution towards a steady state of the Cahn-Hilliard equation is performed by following the well established method in the constant mobility case (e.g., \cite{AW2007}). This hinges upon the regularity of the weak solutions, the presence of a strict Lyapunov function, as well as the 
\L ojasiewicz-Simon inequality.

Beyond the novelty of our results, we point out that the methods herein for the uniqueness and the propagation of regularity of weak solutions are robust. Specifically, Theorem \ref{Goal_thm} can be proven in presence of of polynomial potentials, such as $\Psi_0(s)=(1-s^2)^2$, which is widely used in the literature (see \cite{MIR2019} and the references therein). Furthermore, our analysis can be applied to address other formulations of the Cahn-Hilliard equation with non-degenerate mobility, including viscous regularizations \cite{CGM2024, MZ}, multi-component cases \cite{GGPS2023}, and on evolving surfaces \cite{CE2021, CEGP2023}. Additionally, we expect the application of our arguments to more complex diffuse interface models, such as Navier-Stokes-Cahn-Hilliard systems \cite{ADG2013-2, AGG2024, GG2010, GMT2019, GT2020},  Hele-Shaw/Darcy-Cahn-Hilliard models \cite{G2020, GGW2018}, Brinkman-Cahn-Hilliard systems \cite{CG, EG2019}, and non-isothermal models \cite{EGS, ERS}.

\section{Mathematical setting}
\label{S-Mathset}

\subsection{Notation and function spaces} Let $X$ be a Banach space. 
We denote by $X'$ its dual space and by $\l \cdot , \cdot \r_{X',X}$ the duality product. Given $ 1 \le p \le \infty$ and an interval $I \subseteq [0, \infty)$, we denote with $L^p( I; X)$ the space of all Bochner $p-$integrable functions defined from $I$ into $X$. We also define the space $W^{1,p}(I; X)$ as the set of functions $f \in  L^p( I; X)$ with distributional derivative $ \partial_t f \in L^p(I; X)$. In particular $H^1(I; X) = W^{1,2}(I; X)$. Moreover, $L^p_{\rm uloc}([0,\infty), X)$ denotes the space of functions $f \in L^p(0,T;X)$ for any $T>0$ such that there exists $C>0$ such that 
$$\norm{ f }_{L^p_{\rm uloc}([0,\infty), X)} := 
\sup_{t \ge 0} \int_t^{t+1}\norm{f(s)}_X^p \, \d s <\infty.
$$
For any $a>0$, the set $BC([a,\infty); X)$ is the Banach space of all bounded and continuous functions $f: [a,\infty) \to X$ equipped with the supremum norm, and $BUC([a,\infty); X)$ is the subspace of all bounded and uniformly continuous functions.

 Let $\Omega$ be a bounded smooth domain in $\R^2$. For any positive real $k$ and $1\leq p \leq \infty$, we denote by $W^{k,p}(\Omega)$ the Sobolev space of function in $L^p(\Omega)$ such that the distributional derivative of order up to $k$ is an element of $L^p(\Omega)$.  We use the notation $H^k(\Omega)$ for the Hilbert space $W^{k,2}(\Omega)$ with norm $\norm{\cdot}_ {H^k(\Omega)}$. In particular, for $k=1$, we denote the duality product between $H^1(\Omega)'$ and $H^1(\Omega)$ by $\l \cdot, \cdot \r$. 
We recall that $H^1(\Omega)$ is continuously embedded in $L^r(\Omega)$ for all $r\geq 1$, and the following critical Gagliardo-Nirenberg-Sobolev interpolation inequality holds true (see \cite[Proposition and Remark 1]{O}): there exists a universal constant $C>0$ such that, for all $2 \leq r <\infty$,
\begin{equation}
\label{GN}
\| f\|_{L^r(\Omega)}\leq C \sqrt{r} \| f\|_{L^2(\Omega)}^\frac{2}{r} \| f \|_{H^1(\Omega)}^{\frac{r-2}{r}}, \quad  \forall \, f \in H^1(\Omega).
\end{equation} 
As a consequence, there exists $C>0$ such that
\begin{equation}
\label{CI}
\norm{f }_{H^1(\Omega)'} 
\le C \sqrt{\frac{r}{r-1}} \norm{f}_{L^r(\Omega)}\!, \quad \forall \, f \in L^r(\Omega),\  r>1.
\end{equation}
\begin{remark}
\label{BBm}
Notice that the term $\sqrt{\frac{r}{r-1}}$ was neglected in \cite{BB1999}, and in all the inequalities depending on \eqref{GN}.
\end{remark}
\noindent
Given the definition of total mass
\begin{equation*}
\overline{f}= \frac{1}{|\Omega|} \int_{\Omega} f(x) \, \d x \quad \text{if } f \in L^1(\Omega),
\end{equation*} 
we recall the Poincar\'{e}-Wirtinger inequality
\begin{equation}\label{poincare} 
\norm{f}_{L^2(\Omega)} \le C_P \left( \norm{\nabla f}_{L^2(\Omega)} + |\overline{f}| \right)\!, \quad \forall \, f \in H^1(\Omega),
\end{equation}
where $C_P$ only depends on $\Omega$.
For any $m \in \mathbb{R}$, we introduce the function space
$$
H^1_{(m)}(\Omega)=\left\lbrace f \in H^1(\Omega) : \ \overline{f}=m \right\rbrace \! .
$$
If $m=0$, $H^1_{(0)}(\Omega)$ is a Hilbert space endowed with inner product $(f,g)_{H^1_{(0)}(\Omega)}=(\nabla f, \nabla g)$ and corresponding norm $\| f\|_{H^1_{(0)}(\Omega)}=\| \nabla f \|_{L^2(\Omega)}$.
Extending the definition of total mass as
\begin{equation*}
\overline{f} = \frac{1}{|\Omega|} \l f,1 \r \quad \text{if } f \in H^1(\Omega)',
\end{equation*}
we define the Hilbert space
\begin{equation}
\label{Spazi-mn} 
H^{-1}_{(0)}(\Omega)=\left\lbrace f  \in H^1(\Omega)' : \ \overline{f}=0 \right\rbrace \!,
\end{equation}
with corresponding norm defined as $\| f\|_{H^{-1}_{(0)}(\Omega)}:= \| f\|_{H^1(\Omega)'}$.

\subsection{Elliptic problems.} We now consider the homogeneous Neumann boundary problem for the Laplace operator
\begin{equation}
\label{NP}
\begin{cases}
-\Delta u = f \ & \text{in} \ \Omega \\
\partial_\n u = 0 \ & \text{on} \ \partial\Omega.
\end{cases}
\end{equation}
The solution operator
$\G: H^{-1}_{(0)}(\Omega) \to H^{1}_{(0)}(\Omega)$ is defined as follows: for every $f \in H^{-1}_{(0)}(\Omega)$,  $\G f \in H^{1}_{(0)}(\Omega)$ is the unique function satisfying
\begin{equation}\label{Gdef}
(\nabla \G f, \nabla v) = \l f, v \r, \quad \forall \, v \in H^{1}_{(0)}(\Omega).
\end{equation}
For any $f \in H^{-1}_{(0)}(\Omega)$, $\norm{\nabla \G f}_{L^2(\Omega)}$ is a norm on $H^{-1}_{(0)}(\Omega)$, that is equivalent to $\| f\|_{H^{-1}_{(0)}(\Omega)}$.
Besides,  
\begin{equation}\label{interp}
| \l f, v \r | \leq  \norm{\nabla \G f}_{L^2(\Omega)} \norm{\nabla v}_{L^2(\Omega)}, \quad \forall \, f \in H^{-1}_{(0)}(\Omega), v \in H^{1}_{(0)}(\Omega).
\end{equation}
Next, assuming \eqref{m-ndeg} and given a measurable function $q: \Omega \to [-1,1]$, we introduce the following elliptic problem
\begin{equation}\label{epgq}
\begin{cases}
-\div(b(q) \nabla u) = f \quad & \text{in}\ \Omega \\
b(q) \partial_\n u = 0 \quad & \text{on} \ \partial\Omega.
\end{cases}
\end{equation} 
Similarly to the definition of $\G$ in \eqref{Gdef}, we define the solution operator $\G_q:  H^{-1}_{(0)}(\Omega)  \to  H^{1}_{(0)}(\Omega)$ as follows: 
for every $f \in H^{-1}_{(0)}(\Omega)$,  $\G_q f \in H^{1}_{(0)}(\Omega)$ is the unique function satisfying
\begin{equation}\label{def G_q}
( b(q) \nabla \G_q f ,\nabla v ) = \l f, v \r, \quad \forall \, v \in  H^{1}_{(0)}(\Omega).
\end{equation}
For the readers' convenience, we report some results on the operator $\G_q$ obtained in \cite{BB1999} (see also \cite{G2020, S2007}). From the definition of $\G$ and $\G_q$, for all $f \in  H^{-1}_{(0)}(\Omega)$, we have
\begin{equation*}
\norm{\nabla \G f}_{L^2(\Omega)}^2 
= \l f , \G f \r = (b(q) \nabla \G_q f , \nabla \G f) 
\le \sqrt{b_M}  \norm{ \sqrt{b(q)} \nabla \G_q f}_{L^2(\Omega)} \norm{\nabla \G f}_{L^2(\Omega)}\!,
\end{equation*}
and
\begin{equation*}
\norm{\sqrt{b(q)} \nabla \G_q f}_{L^2(\Omega)}^2 
= \l f, \G_q f\r = ( \nabla \G f , \nabla \G_q v) 
\le \frac{1}{\sqrt{b_m}} \norm{\nabla \G f}_{L^2(\Omega)} 
\norm{\sqrt{b(q)} \nabla \G_q f}_{L^2(\Omega)} \!.
\end{equation*}
Combining the last two inequalities yields
\begin{equation}\label{equivalence of norm}
\sqrt{b_m} \norm{\sqrt{b(q)} \nabla \G_q f}_{L^2(\Omega)}
\le \norm{ \nabla \G f}_{L^2(\Omega)}
\le \sqrt{b_M}\norm{\sqrt{b(q)} \nabla \G_q f}_{L^2(\Omega)}\!, \quad \forall \, f \in  H^{-1}_{(0)}(\Omega).
\end{equation}
Thus, $\norm{\nabla \G_q f}_{L^2(\Omega)}$ is a norm on $ H^{-1}_{(0)}(\Omega)$ that is equivalent to $\norm{\nabla \G f}_{L^2(\Omega)}$.
Also, we observe that 
\begin{equation}
\label{AUTO}
\l f, \G_q  g \r = ( b(q) \nabla \G_q f, 	\nabla \G_q g)=  (\nabla \G_q f, b(q) \nabla \G_q g)
= \l g, \G_q f \r, \quad \forall \, f,g \in  H^{-1}_{(0)}(\Omega).
\end{equation}
Besides, we have 
\begin{equation}
\label{inter-2}
\| f\|_{L^2(\Omega)}= \sqrt{\left( \nabla \G_q f, \nabla f\right)} \leq 
\| \nabla \G_q f \|_{L^2(\Omega)}^\frac12 \| \nabla f \|_{L^2(\Omega)}^\frac12,
\quad \forall \, f \in  H^{1}_{(0)}(\Omega).
\end{equation}
We now derive some elliptic estimates related to problem \eqref{epgq}. Let $f \in L^p(\Omega)	\cap  H^{-1}_{(0)}(\Omega)$, $p>1$. Assuming that $b\in C^1([-1,1])$ and $q \in H^2(\Omega)$,  we choose $v = \frac{w}{b(q)} - \overline{\frac{w}{b(q)}}$ for any $w \in  H^{1}_{(0)}(\Omega)$ (note that $v \in  H^{1}_{(0)}(\Omega)$)  in \eqref{def G_q} obtaining
\begin{equation}\label{weak-2P}
\int_{\Omega}  \nabla \G_q f \cdot \nabla w \, \d x = \bigintssss_{\Omega} \left( \frac{b'(q)}{b(q)}\nabla q \cdot  \nabla \G_q f  + \frac{f}{b(q)} \right)  w \, \d x =: (g, w).
\end{equation}
Therefore, by the classical regularity theory for problem \eqref{NP}, if $ \nabla \G_q f \cdot \nabla q \in L^p(\Omega)$, we have 
\begin{equation}
\label{W2r}
\| \G_q f\|_{W^{2,p}(\Omega)} \leq C \left( 
\left\|  \nabla q \cdot \nabla \G_q f \right\|_{L^p(\Omega)}
+
\| f\|_{L^p(\Omega)}
 \right)\!, \quad 
\forall \, 1<p<\infty,
\end{equation}
where the positive constant $C$ depends on $p$ and $b$. If $f \in L^2(\Omega)$ such that $\overline{f}=0$, by \cite[Theorem 2.1]{G2020} $\G_q f \in H^2(\Omega)$. Taking $p=2$ in \eqref{W2r}, and by using \eqref{GN} with $r=4$, we find
\begin{align*}
\| \G_q f\|_{H^2(\Omega)} 
&\leq C \left(
\norm{  \nabla q \cdot \nabla \G_q f}_{L^2(\Omega)} 
+  \| f\|_{L^2(\Omega)}  \right) 
\\
& \leq C \left( 
\norm{  \nabla q }_{L^4(\Omega)} 
\norm{\nabla \G_q f}_{L^4(\Omega)} +  \| f\|_{L^2(\Omega)} \right)
\\
& \leq C \left( 
\norm{\nabla q}_{L^2(\Omega)}^\frac12
\norm{q}_{H^2(\Omega)}^\frac12
\norm{  \nabla \G_q f }_{L^2(\Omega)}^\frac12
\norm{ \G_q f }_{H^2(\Omega)}^\frac12
+\| f\|_{L^2(\Omega)} 
\right)\!.
\end{align*}
Then, it follows from the Young inequality that
\begin{equation}
\label{H2}
\| \mathcal{G}_{q} f\|_{H^2(\Omega)}\leq C \left( \| \nabla q\|_{L^2(\Omega)} \| q\|_{H^2(\Omega)} \| \nabla \mathcal{G}_q f\|_{L^2(\Omega)}+ \| f\|_{L^2(\Omega)} \right)\!.
\end{equation}
The above inequality is a particular case of the following one, which is proven in \cite{BB1999} (we report its proof in Appendix \ref{App-0} for completeness)
\begin{equation}\label{H2-BB}
\norm{\G_q f}_{H^2(\Omega)} \le   
\left(\frac{C^2 s^2}{s-2}\right)^{\frac{s}{4}} 
\norm{\nabla q}_{L^2(\Omega)}^{\frac{s-2}{2}} 
\norm{q}_{H^2(\Omega)} \norm{\nabla \G_q f}_{L^2(\Omega)}
+ \frac{s C}{2} \norm{f}_{L^2(\Omega)}, \quad \forall \, s \in (2,\infty),
\end{equation}
where the positive constant $C$ is independent of $s$. Furthermore, if $f \in L^4(\Omega)$ with $\overline{f}=0$, $g $ in \eqref{weak-2P} belongs to $L^4(\Omega)$ since $\G_q f \in H^2(\Omega)$ and $\nabla q \in H^1(\Omega)$. Choosing $p=4$ in \eqref{W2r}, and exploiting \eqref{GN} with $r=8$, we obtain
\begin{align}
\| \mathcal{G}_q f\|_{W^{2,4}(\Omega)}
&\leq 
C \left( 
 \norm{\nabla q}_{L^8(\Omega)} \norm{\nabla \mathcal{G}_q f}_{L^8(\Omega)} + \| f\|_{L^4(\Omega)} \right) \notag
 \\
 &\leq C \left(  \norm{\nabla q}_{L^2(\Omega)}^\frac14 \norm{q}_{H^2(\Omega)}^\frac34 \norm{\nabla \mathcal{G}_q f}_{L^2(\Omega)}^\frac14  
 \norm{\mathcal{G}_q f}_{H^2(\Omega)}^\frac34 
 +\| f\|_{L^4(\Omega)}  
  \right)\!.
  \label{W24}
\end{align}
Finally, if $f \in  H^{1}_{(0)}(\Omega)$ and $b \in C^2([-1,1])$, the elliptic regularity theory entails that $\G_q f \in H^3(\Omega)$. In particular, there exists $C>0$ such that 
\begin{equation}
\label{H3}
\| \G_q f\|_{H^3(\Omega)} \leq C
\left( 
\norm{  \frac{b'(q)}{b(q)} \nabla q \cdot \nabla \G_q f}_{H^1(\Omega)}
+
\norm{ \frac{f}{b(q)} }_{H^1(\Omega)}
 \right)\!.
\end{equation}

\section{Proof of Theorem \ref{Goal_thm} - (A) - (B): 
Existence and Uniqueness of Weak Solutions}
\label{S-WEAK}

This section is devoted to show Theorem \ref{Goal_thm} - (A) - (B). The proof is divided in two parts.
\medskip

\textbf{Properties of weak solutions.} Let $\varphi$ be a global weak solution given by Theorem \ref{BB} - (1). First of all, since $\vphi \in L^2(0,T; H^2(\Omega))$ for any $T>0$, it easily follows that $\partial_\n \vphi=0$ almost everywhere on $\partial \Omega \times (0,\infty)$,  and $\mu =-\Delta \vphi+\Psi'(\vphi)$ almost everywhere in $\Omega \times (0,\infty)$. Since $\overline{\mu}= \overline{\Psi'(\varphi)}$ from \eqref{wCH2}, we can rewrite \eqref{wCH1} as
\begin{equation}
\label{mu&der}
\mu= - \G_{\vphi} \partial_t \vphi + \overline{\Psi' (\vphi)}, \quad \text{a.e. in } 
\Omega \times (0,\infty).
\end{equation}
Moreover, taking $v=1$ in \eqref{wCH1}, we derive the conservation of mass, namely 
\begin{equation}
\label{cons-mass}
\overline{\varphi}(t)=\overline{\varphi_0}=m, \quad \forall \, t \geq 0.
\end{equation}

Let us now define the functional $E_0:L^2(\Omega)\rightarrow \R $ given by
\begin{equation}
\label{E_0}
E_0( \varphi)= \frac12 \| \nabla \varphi\|_{L^2(\Omega)}^2
+ \int_{\Omega} F(\varphi)\, \d x,
\end{equation}
where $F$ is the convex part of $\Psi$ as in \eqref{f-log}.
The functional $E_0$ is proper, convex and lower-semicontinuous. Owing to \cite[Lemma 4.1]{RS}, we deduce that $t\mapsto E_0(\varphi(t))$ is absolutely continuous on $[0,\infty)$ and
$$
\frac{\d}{\d t} E_0(\varphi)=\l \partial_t \varphi, -\Delta \vphi+F'(\vphi) \r= \l\partial_t \varphi ,\mu+\Theta_0\varphi\r,
$$
almost everywhere on $(0,\infty)$. As a consequence, for any $T>0$, due to the boundedness of $F$ and $\varphi \in C([0,T]; L^2(\Omega))$, the Lebesgue theorem entails that
$
\int_{\Omega} F(\varphi(\cdot)) \ \d x \in C([0,T]),
$
which in turn gives $\varphi \in C([0,T]; H^1(\Omega))$.
Now, taking $v=\mu$ in \eqref{wCH1} and exploiting the standard chain rule in $L^2(0,T;H^1(\Omega))\cap H^1(0,T;H^1(\Omega)')$, we obtain
\begin{equation}
\label{eiCahn}
\ddt E(\varphi) + \int_{\Omega} b(\varphi) |\nabla \mu|^2 \, \d x =0,
\end{equation}
which implies the energy equality \eqref{EE}. 
Then, we deduce from \eqref{poincare} and \eqref{cons-mass} that
\begin{equation}
\label{H1-phi}
\sup_{t\geq 0} \| \varphi(t)\|_{H_{(m)}^1(\Omega)} \leq C,
\end{equation}
and
\begin{equation}
\label{mu-0inf}
\int_0^\infty \| \nabla \mu (s)\|_{L^2(\Omega)}^2 \, \d s
\leq C,
\end{equation}
for some constant $C$ depending on $E(\varphi_0)$ and $m$. The latter 
implies $\nabla \mu \in L^2(0,\infty; L^2(\Omega))$. By \eqref{equivalence of norm} and \eqref{mu&der}, we have 
\begin{equation}
\label{phit-mu}
\| \nabla \G \partial_t \varphi \|_{L^2(\Omega)} \leq C \| \nabla \mu\|_{L^2(\Omega)}.
\end{equation}
Then, \eqref{mu-0inf} yields
\begin{equation}
\label{phit-0inf}
\int_0^\infty \| \partial_t \varphi (s)\|_{ H^{-1}_{(0)}(\Omega)}^2 \, \d s
\leq C,
\end{equation}
which gives $\partial_t \vphi \in L^2(0,\infty;  H^{-1}_{(0)}(\Omega))$. 

Next, multiplying \eqref{wCH2} by $-\Delta \varphi$, integrating over $\Omega$ and exploiting the boundary condition on $\varphi$, we find
\begin{equation}
\label{muDeltaphi}
\| \Delta \vphi\|_{L^2(\Omega)}^2 + \int_{\Omega} F'(\vphi) \left( -\Delta \vphi \right) \, \d x
= \int_{\Omega} \nabla \mu \cdot \nabla \vphi\, \d x + \Theta_0 \| \nabla \vphi\|_{L^2(\Omega)}^2.
\end{equation}
It is well known that the second term on the left-hand side is non-negative by the convexity of $F$ (see, e.g., \cite[Appendix A]{CG}). Then, by \eqref{H1-phi}, it follows that
\begin{equation}
\label{H2-phi-1}
\| \Delta \vphi\|_{L^2(\Omega)}^2 \leq C \left( 1+ \| \nabla \mu\|_{L^2(\Omega)}\right)\!.
\end{equation}  
Thus, by the standard elliptic regularity theory, we have
\begin{equation}
\label{H2-phi-2}
\sup_{t \geq 0} \int_t^{t+1} \|\varphi (s)\|_{H^2(\Omega)}^4 \, \d s 
\leq C,
\end{equation}
which proves that $\varphi \in L_{\uloc}^4([0,\infty);H^2(\Omega))$. We proceed with the control of the total mass of $\mu$. We report the standard tool from \cite{MZ}: there exists a positive constant $\overline{C}$, depending only on $m=\overline{\varphi_0}$, such that
\begin{equation}
\label{MZ}
\int_\Omega| F'(\varphi)|\, \d x\leq \overline{C} 
\left|\int_\Omega F'(\varphi)(\varphi-\overline{\varphi} )\, \d x\right|+\overline{C},
\end{equation}
where $\overline{C} \to +\infty$ as $|\overline{\varphi}_0|\to 1$. 
Multiplying $\mu$ by $\varphi-\overline{\varphi}$ and integrating over $\Omega$, we easily deduce that
\begin{align*}
\| \nabla \vphi\|_{L^2(\Omega)}^2 +
\int_{\Omega} F'(\varphi) \left( \varphi-\overline{\varphi} \right) \, \d x
= \int_\Omega \left( \mu -\overline{\mu}\right) \vphi \, \d x 
+ \Theta_0 \int_\Omega \vphi (\vphi-\overline{\vphi}) \, \d x.
\end{align*}
By \eqref{poincare} and \eqref{H1-phi}, we get
\begin{align*}
\left|
\int_{\Omega} F'(\varphi) \left( \varphi-\overline{\varphi} \right) \, \d x
\right|
\leq C \|\nabla \mu \|_{L^2(\Omega)} + C,
\end{align*}
where the positive constant $C$ depends on $E(\varphi_0)$ and $m$.
Since
$
|\overline{\mu}| = \overline{\Psi'(\vphi)},
$
\eqref{MZ} and the above inequality imply that
\begin{equation}
\label{mediamu}
|\overline{\mu}|\leq C \left( 1+ \| \nabla \mu\|_{L^2(\Omega)}\right)\!.
\end{equation}
In light of \eqref{mu-0inf}, we conclude that $\mu \in L^2_{\uloc}([0,\infty);H^1(\Omega))$. Finally, noticing that
\begin{equation*}
\begin{cases}
-\Delta\varphi+F'(\varphi)=\mu^* \quad &\text{in }\Omega\\
\partial_{\n}\varphi=0\quad &\text{on }\partial\Omega,
\end{cases}
\end{equation*}
almost everywhere in time on $(0,\infty)$,
where $\mu^*=\mu+\Theta_0\varphi \in L^2_{\uloc}([0,\infty);H^1(\Omega))$, 
an application of \cite[Lemma A.4]{CG} yields
\begin{equation}
\label{W2p}
\| \vphi\|_{W^{2,p}(\Omega)}+ \| F'(\vphi)\|_{L^p(\Omega)} \leq C(p)\left( 1+ \| \mu\|_{H^1(\Omega)}\right)\!, \quad \forall \, p \in [2,\infty),
\end{equation}
which provides that $\vphi \in L^2_{\uloc}([0,\infty); W^{2,p}(\Omega))$ and $F'(\vphi)\in L^2_{\uloc}([0,\infty); L^{p}(\Omega))$, for any $p \in [2,\infty)$.
\medskip

\textbf{Uniqueness of weak solutions.} Let $\vphi_1$ and $\vphi_2$ be two weak solutions to problem \eqref{CH1}-\eqref{nCH-mu} with non-degenerate mobility \eqref{m-ndeg} as in Theorem \ref{Goal_thm} - part (A), associated with two initial data $\vphi_{1}^0$ and $\vphi_2^0$ such that $\overline{\vphi_{1}^0}=\overline{\vphi_{2}^0}$. We observe that \eqref{wCH1} and \eqref{wCH2} entail that 
\begin{equation}
\label{mu-G}
\mu_i= - \mathcal{G}_{\vphi_i} \partial_t \vphi_i + \overline{\Psi'(\vphi_i)}, \quad 
\text{ in } \Omega, \ \text{a.e. in } (0,\infty), \ i=1,2,
\end{equation}
where $\mu_1$ and $\mu_2$ are the chemical potentials corresponding to $\vphi_1$ and $\vphi_2$, respectively. 
We define $\Phi=\vphi_1-\vphi_2$, which solves
\begin{equation*}
-\Delta \Phi + \Psi'(\vphi_1)-\Psi'(\vphi_2)= \mu_1-\mu_2,  \quad 
\text{ in } \Omega, \ \text{a.e. in } (0,\infty).
\end{equation*}
We preliminary notice that $\overline{\Phi}\equiv 0$ by the conservation of mass. Multiplying the above equation by $\Phi$ and integrating over $\Omega$, we find
 \begin{equation*}
\norm{\nabla \Phi}_{L^2(\Omega)}^2 +  \int_\Omega \left( F'(\vphi_1)-F'(\vphi_2) \right)   \Phi \, \d x 
- \int_{\Omega} \left( \mu_1-\mu_2\right)  \Phi \, \d x= \Theta_0 \norm{\Phi}_{L^2(\Omega)}^2\!.
\end{equation*}
Thanks to \eqref{mu-G}, we rewrite the above as 
\begin{equation*}
\norm{\nabla \Phi}_{L^2(\Omega)}^2 +\int_\Omega \left( F'(\vphi_1)-F'(\vphi_2) \right)   \Phi \, \d x  
+ \int_\Omega \left( \G_{\vphi_1} \partial_t \vphi_1 - \G_{\vphi_2} \partial_t \vphi_2 \right)  \Phi \, \d x= \Theta_0 \norm{\Phi}_{L^2(\Omega)}^2\!.
\end{equation*}
We notice that
\begin{equation*}
\begin{split}
 \int_\Omega \left( \G_{\vphi_1} \partial_t \vphi_1 - \G_{\vphi_2} \partial_t \vphi_2 \right)  \Phi \, \d x
&= \int_\Omega \G_{\vphi_1} \partial_t \Phi \,  \Phi \, \d x 
+ \int_\Omega (\G_{\vphi_1}-\G_{\vphi_2}) \partial_t \vphi_2 \,  \Phi \, \d x. 
\end{split}
\end{equation*}
Next, we claim that 
\begin{equation}
\label{IP-time}
\begin{split}
\int_\Omega \G_{\vphi_1} \partial_t \Phi \,  \Phi \, \d x 
&= \ddt \frac{1}{2} (\G_{\vphi_1} \Phi, \Phi) 
+\frac12 \int_\Omega \nabla \G \partial_t \vphi_1 \cdot b''(\vphi_1) \nabla \vphi_1 \left| \nabla \G_{\vphi_1} \Phi \right|^2 \, \d x
\\[5pt]
&\quad +  \int_\Omega \nabla \G \partial_t \vphi_1 \cdot  b'(\vphi_1) \left( D^2 \G_{\vphi_1} \Phi \nabla \G_{\vphi_1} \Phi \right) \, \d x,
\end{split}
\end{equation}
almost everywhere in $(0,\infty)$. Here, $D^2 f$ is the Hessian of $f$.
We recall that 
$$
(\G_{\vphi_1} \Phi, \Phi)^\frac12= \norm{\sqrt{b(\vphi_1)}\nabla \G_{\vphi_1} \Phi}_{L^2(\Omega)}
$$ 
is a norm on $ H^{-1}_{(0)}(\Omega)$, that is equivalent to
$ \norm{\nabla \G \Phi}_{L^2(\Omega)}$.
Then, since $(F'(\vphi_1)-F'(\vphi_2), \varphi_1-\varphi_2) \geq 0$, we infer that
\begin{equation}\label{uniqueness estimate}
\begin{split}
&\ddt  \frac12 \norm{\sqrt{b(\vphi_1)} \nabla \G_{\vphi_1} \Phi}_{L^2(\Omega)}^2 + \norm{\nabla \Phi}_{L^2(\Omega)}^2 \leq  \Theta_0 \norm{\Phi}_{L^2(\Omega)}^2+ I_1 +I_2+ I_3,
\end{split}
\end{equation}
where 
\begin{align}
\label{I1}
I_1 &=- \frac12 \int_\Omega \nabla \G \partial_t \vphi_1 \cdot b''(\vphi_1) \nabla \vphi_1 \left| \nabla \G_{\vphi_1} \Phi \right|^2 \, \d x,
\\
\label{I2}
I_2 &=-  \int_\Omega \nabla \G \partial_t \vphi_1 \cdot  b'(\vphi_1) \left( D^2 \G_{\vphi_1} \Phi \nabla \G_{\vphi_1} \Phi \right) \, \d x,
\end{align}
and
\begin{equation}
\label{I3}
I_3=-  \int_\Omega (\G_{\vphi_1}-\G_{\vphi_2}) \partial_t \vphi_2 \,  \Phi \, \d x.
\end{equation}
The rest of the proof consists of two parts: the justification of \eqref{IP-time} and the control of $I_1$, $I_2$ and $I_3$.
\medskip

\textbf{Proof of \eqref{IP-time}.} 
For any $\varepsilon \in (0,1)$, we consider the elliptic problem 
\begin{equation}
\label{ell-reg}
\begin{cases}
-\varepsilon \Delta \vphi_1^\varepsilon + \vphi_1^\varepsilon = \vphi_1 \quad & \text{in } \Omega,\\
\partial_\n \vphi_1^\varepsilon=0\quad &\text{on } \partial \Omega,
\end{cases}
\end{equation}
almost everywhere in $(0,\infty)$. Recalling that $\vphi_1 \in L^\infty(0,T; H^1(\Omega))\cap L^4(0,T; H^2(\Omega))$ with $\partial_t \varphi \in L^2(0,T;  H^{-1}_{(0)}(\Omega))$, for any $T>0$, it follows by the elliptic regularity theory that 
\begin{equation}
\label{vphieps}
\vphi_1^\varepsilon \in L^\infty(0,T; H^3(\Omega)), \quad  \partial_t \varphi_1^\varepsilon \in L^2(0,T;  H^{1}_{(0)}(\Omega)), 
\end{equation}
and $\partial_\n \vphi_1^\varepsilon =0$ on $\partial \Omega$  almost everywhere in $(0, T)$, for any $\varepsilon\in (0,1)$ and any $T>0$.
In addition, recalling that $\vphi_1 \in L^\infty(\Omega \times (0,T))$ with $|\vphi_1|<1$ almost everywhere in $\Omega \times (0,T)$, we also have  
\begin{equation}
\label{Linf-eps}
\vphi_1^\varepsilon \in L^\infty(\Omega \times (0,T)) \ : 
|\vphi_1^\varepsilon|<1 \quad \text{a.e. in } \Omega \times (0,T).
\end{equation}
Besides, it is easy to check that
\begin{align}
\label{eps-1}
\int_0^T \left\| \vphi_1^\varepsilon(s) \right\|_{L^2(\Omega)}^4 + \| \nabla \vphi_1^\varepsilon(s)\|_{L^2(\Omega)}^4 \, \d s 
&\leq
\int_0^T \| \vphi_1(s)\|_{L^2(\Omega)}^4 + \| \nabla \vphi_1(s)\|_{L^2(\Omega)}^4  \, \d s, 
\\
\label{eps-3}
\int_0^T \| \nabla \G \partial_t \vphi_1^\varepsilon (s)\|_{L^2(\Omega)}^2  \, \d s
&\leq
\int_0^T \| \nabla \G \partial_t \vphi_1 (s)\|_{L^2(\Omega)}^2 \, \d s,
\end{align}
and 
\begin{equation}
\label{eps-2}
\| \vphi_1^\varepsilon\|_{L^\infty(0,T;H^1(\Omega))}
\leq \| \vphi_1\|_{L^\infty(0,T;H^1(\Omega))}.
\end{equation}
Owing to the regularity of $\varphi_1^\varepsilon$, we compute the gradient of \eqref{ell-reg} and test it by $-\nabla \Delta \vphi_1^\varepsilon$. Exploiting the Neumann boundary condition on $\varphi_1^\varepsilon$ and $\vphi_1$, we find 
$$
\varepsilon \| \nabla \Delta \vphi_1^\varepsilon\|_{L^2(\Omega)}^2+
 \| \Delta \vphi_1^\varepsilon\|_{L^2(\Omega)}^2 = (\Delta \vphi_1,\Delta \vphi_1^\varepsilon),
$$
which gives 
\begin{equation}
\label{eps-4}
\int_0^T \| \Delta \vphi_1^\varepsilon(s)\|_{L^2(\Omega)}^4  \, \d s\leq
\int_0^T \| \Delta \vphi_1 (s)\|_{L^2(\Omega)}^4 \, \d s.
\end{equation}
Hence, by the above estimates, we derive that, up to a subsequence, $\varphi_1^\varepsilon \rightharpoonup \varphi_1$ weakly in $L^4(0,T;H^2(\Omega))$ and weak-star in $L^\infty(\Omega \times (0,T))$, as well as 
$\partial_t \varphi_1^\varepsilon \rightharpoonup \partial_t \varphi_1$ weakly in $L^2(0,T;H^{-1}_{(0)}(\Omega))$. In order to recover the strong convergence, we report the following result which is based on the Clarkson inequality: 
\begin{lemma}
\label{lemmaws}
Let $X$ be a Banach space and let $p \in [2, \infty)$. Assume that  $u_n \rightharpoonup u$ weakly in $L^p(0,T; X)$ and $\limsup_{n \to \infty} \| u_n\|_{L^p(0,T; X)}\leq \| u\|_{L^p(0,T; X)}$. Then, $u_n \to u$ strongly in $L^p(0,T; X)$.
\end{lemma}

\noindent
In light of \eqref{eps-3} and \eqref{eps-4}, Lemma \ref{lemmaws} implies that
\begin{equation}
\label{conv-eps}
\vphi_1^\varepsilon \to \vphi_1 \quad \text{strongly in } L^4(0,T;H^2(\Omega)), \quad
\partial_t \vphi_1^\varepsilon \to \partial_t \vphi_1 \quad \text{strongly in } L^2(0,T;  H^{-1}_{(0)}(\Omega)).
\end{equation}
In particular, up to a subsequence, we have
\begin{equation}
\label{eps-point}
\vphi_1^\varepsilon \to  \vphi_1, \quad \nabla \vphi_1^\varepsilon \to \nabla \vphi_1 \quad  \text{a.e. in } \Omega \times (0,T).
\end{equation}
Thanks to \eqref{Linf-eps}, we now define the sequence $\G_{\vphi_1^\varepsilon}\Phi \in L^\infty(0,T;H^1(\Omega))$.  It easily follows from \eqref{eps-2} that 
\begin{equation}
\label{G-H1}
\norm{  \G_{\vphi_1^\varepsilon}\Phi }_{L^\infty(0,T; H^{1}_{(0)}(\Omega))}\leq C,
\end{equation}
where $C$ is  independent of $\varepsilon$. In light of \eqref{eps-2} and \eqref{G-H1}, an application of \eqref{H2} entails that 
$$
\norm{ \G_{\vphi_1^\varepsilon}\Phi }_{H^2(\Omega)}
\leq C \left( 1+ \| \varphi_1^\varepsilon \|_{H^2(\Omega)}\right)\!,
$$
for some constant $C$ independent of $\varepsilon$. Thus, we infer from \eqref{eps-2} and \eqref{eps-4} that 
\begin{equation}
\label{G-H2}
\int_0^T \left\| \G_{\vphi_1^\varepsilon}\Phi (s) \right\|_{H^2(\Omega)}^4 \, \d s \leq C,
\end{equation}
namely $\G_{\vphi_1^\varepsilon}\Phi$ is uniformly bounded in $L^4(0,T;H^2(\Omega))$. Furthermore, by \eqref{W24}, we also learn that 
\begin{equation}
\norm{ \G_{\vphi_1^\varepsilon}\Phi }_{W^{2,4}(\Omega)}
\leq C \left( 1+ \norm{ \varphi_1^\varepsilon }_{H^2(\Omega)}^\frac34 
\norm{\G_{\vphi_1^\varepsilon}\Phi }_{H^2(\Omega)}^\frac34
\right)\!.
\end{equation}
Hence, integrating over $(0,T)$, we obtain
\begin{equation}
\label{G-W24}
\int_0^T \| \G_{\vphi_1^\varepsilon}\Phi (s)\|_{W^{2,4}(\Omega)}^\frac83 \, \d s 
\leq CT + C\int_0^T \norm{ \varphi_1^\varepsilon (s) }_{H^2(\Omega)}^4+
\norm{ \G_{\vphi_1^\varepsilon}\Phi (s) }_{H^2(\Omega)}^4 \, \d s\leq C,
\end{equation}
for some $C$ independent of $\varepsilon$. The latter implies that $\G_{\vphi_1^\varepsilon}\Phi$ is uniformly bounded in $L^\frac83(0,T;W^{2,4}(\Omega))$.

Now, we study the convergence properties of the operator $\G_{\vphi_1^\varepsilon}$ as $\varepsilon \to 0$. Let $f \in L^2(0,T;  H^{-1}_{(0)}(\Omega))$. By definition \eqref{def G_q}, we know that 
$$
\left( b(\vphi_1) \nabla \G_{\vphi_1} f, \nabla v \right)=
\left( b(\vphi_1^\varepsilon) \nabla \G_{\vphi_1^\varepsilon} f, \nabla v \right),
\quad \forall \, v \in  H^{1}_{(0)}(\Omega), 
$$
which gives
$$
\left( b(\vphi_1) \nabla \left( \G_{\vphi_1} f- \G_{\vphi_1^\varepsilon} f \right), \nabla v \right)=
\left( \left( b(\vphi_1^\varepsilon) - b(\vphi_1)\right) \nabla \G_{\vphi_1^\varepsilon} f, \nabla v \right),
\quad \forall \, v \in  H^{1}_{(0)}(\Omega).
$$
Therefore, we have 
\begin{align*}
\int_0^T \norm{ \nabla \left( \G_{\vphi_1} f- \G_{\vphi_1^\varepsilon} f \right) }_{L^2(\Omega)}^2 \, \d s
&\leq \frac{1}{b_m} \int_0^T \norm{ \left( b(\vphi_1^\varepsilon) - b(\vphi_1)\right) \nabla \G_{\vphi_1^\varepsilon} f }_{L^2(\Omega)}^2 \, \d s
\\
&\leq \frac{1}{b_m} \int_0^T \norm{ b(\vphi_1^\varepsilon) - b(\vphi_1) }_{L^\infty(\Omega)}^2 \norm{ \nabla \G_{\vphi_1^\varepsilon} f }_{L^2(\Omega)}^2 \, \d s.
\end{align*}
We aim to use the Lebesgue convergence theorem on the right-hand side.
We observe from \eqref{equivalence of norm} that 
$$
\norm{ b(\vphi_1^\varepsilon) - b(\vphi_1) }_{L^\infty(\Omega)}^2 
\norm{ \nabla \G_{\vphi_1^\varepsilon} f }_{L^2(\Omega)}^2  
\leq
\frac{4 b_M}{b_m}  \norm{ \nabla \G f }_{L^2(\Omega)}^2, 
$$
where $\| \nabla \G f \|_{L^2(\Omega)}^2 \in L^1(0,T)$. 
On the other hand,
$$
\norm{ b(\vphi_1^\varepsilon) - b(\vphi_1) }_{L^\infty(\Omega)}^2 
\leq C \norm{ \vphi_1^\varepsilon - \vphi_1 }_{L^\infty(\Omega)}^2 
\leq C  \norm{ \vphi_1^\varepsilon - \vphi_1}_{H^2(\Omega)}^2 
$$ 
almost everywhere in $(0,T)$. In light of \eqref{conv-eps}, up to a subsequence, we obtain that 
$$
\norm{ b(\vphi_1^\varepsilon) - b(\vphi_1) }_{L^\infty(\Omega)}^2 
\norm{ \nabla \G_{\vphi_1^\varepsilon} f }_{L^2(\Omega)}^2 
\leq 
C  \| \vphi_1^\varepsilon - \vphi_1\|_{H^2(\Omega)}^2 
\| \nabla \G f \|_{L^2(\Omega)}^2 
\to 0
$$
 almost everywhere in $(0,T)$. 
 Thus, up to a subsequence, we conclude that 
\begin{equation*}
\nabla \G_{\vphi_1^\varepsilon} f \to \nabla \G_{\vphi_1} f 
\quad \text{strongly in } L^2(0,T; L^2(\Omega)).
\end{equation*}
In particular, for $f=\partial_t \Phi$ and $f= \Phi$, we have
\begin{equation}
\label{Gt-sH1}
\nabla \G_{\vphi_1^\varepsilon} \partial_t \Phi \to \nabla \G_{\vphi_1} \partial_t \Phi 
\quad \text{strongly in } L^2(0,T; L^2(\Omega)),
\end{equation}
and 
\begin{equation}
\label{G-sH1}
\nabla \G_{\vphi_1^\varepsilon} \Phi \to \nabla \G_{\vphi_1} \Phi 
\quad \text{strongly in } L^2(0,T; L^2(\Omega)),
\end{equation}
respectively. The latter means that $\G_{\vphi_1^\varepsilon} \Phi \to \G_{\vphi_1} \Phi$ strongly in $L^2(0,T;  H^{1}_{(0)}(\Omega))$. Moreover, by the above uniform estimates \eqref{G-H2} and \eqref{G-W24}, we conclude that 
\begin{equation}
\label{conv-G-H2}
\G_{\vphi_1^\varepsilon} \Phi  \rightharpoonup \G_{\vphi_1} \Phi \quad 
\text{weakly in } L^4(0,T;H^2(\Omega)),
\end{equation}
and 
\begin{equation}
\label{conv-G-W24}
\G_{\vphi_1^\varepsilon} \Phi  \rightharpoonup \G_{\vphi_1} \Phi \quad 
\text{weakly in } L^\frac83(0,T;W^{2,4}(\Omega)).
\end{equation}

Next, thanks to \eqref{vphieps}, it is immediate to see that (cf. \cite[Eqn. (2.22)]{BB1999}), for any $\eta \in \mathcal{D}(0,T)$,
\begin{equation}
\label{IP-time-eps}
\begin{split}
\int_0^T (\G_{\vphi_1^\varepsilon} \partial_t \Phi, \Phi) \, \eta \, \d s
&= -\frac{1}{2} \int_0^T (\G_{\vphi_1^\varepsilon} \Phi, \Phi) \, \partial_t \eta \, \d s 
\\
&\quad +\frac12 \int_0^T  \left( \partial_t \vphi_1^\varepsilon, b'(\vphi_1^\varepsilon) \nabla \G_{\vphi_1^\varepsilon} \Phi \cdot \nabla \G_{\vphi_1^\varepsilon} \Phi \right) \, \eta \, \d s.
\end{split}
\end{equation}
We aim to pass to the limit as $\varepsilon \to 0$ in \eqref{IP-time-eps}.
First, we easily conclude from \eqref{Gt-sH1} and \eqref{G-sH1} that
\begin{equation}
\lim_{\varepsilon\to 0} \int_0^T (\G_{\vphi_1^\varepsilon} \partial_t \Phi, \Phi) \, \eta \, \d s
=  \int_0^T (\G_{\vphi_1} \partial_t \Phi, \Phi) \, \eta \, \d s
\end{equation}
and
\begin{equation}
\lim_{\varepsilon \to 0}
-\frac{1}{2} \int_0^T (\G_{\vphi_1^\varepsilon} \Phi, \Phi) \, \partial_t \eta \, \d s
=
 -\frac{1}{2} \int_0^T (\G_{\vphi_1} \Phi, \Phi) \, \partial_t \eta \, \d s.
\end{equation}
In order to handle the last term on the right-hand side of \eqref{IP-time-eps}, we first show that 
\begin{equation}
\label{b-V}
\nabla \left( b'(\vphi_1^\varepsilon) |\nabla \G_{\vphi_1^\varepsilon} \Phi|^2\right)  \text{ is uniformly bounded in } L^2(0,T;L^2(\Omega)) \text{ w.r.t. }\varepsilon.
\end{equation} 
In fact, by \eqref{GN}, we have
\begin{align*}
\left\| \nabla \left( b'(\vphi_1^\varepsilon) |\nabla \G_{\vphi_1^\varepsilon} \Phi|^2\right) \right\|_{L^2(\Omega)}
&\leq 
\norm{ b''(\vphi_1^\varepsilon)  \nabla \vphi_1^\varepsilon |\nabla \G_{\vphi_1^\varepsilon} \Phi |^2}_{L^2(\Omega)} 
+ 2 \norm{ b'(\vphi_1^\varepsilon) D^2 \G_{\vphi_1^\varepsilon} \Phi \nabla \G_{\vphi_1^\varepsilon} \Phi }_{L^2(\Omega)}
\\
&\leq 
C \norm{ \nabla \vphi_1^\varepsilon }_{L^6(\Omega)} 
\norm{ \nabla \G_{\vphi_1^\varepsilon} \Phi }_{L^6(\Omega)}^2 
+ C  \norm{ D^2 \G_{\vphi_1^\varepsilon} \Phi  }_{L^4(\Omega)} 
\norm{ \nabla \G_{\vphi_1^\varepsilon} \Phi }_{L^4(\Omega)}
\\
&\leq 
C \norm{ \nabla \vphi_1^\varepsilon }_{L^2(\Omega)}^\frac13 
\norm{ \nabla \vphi_1^\varepsilon }_{H^1(\Omega)}^\frac23 
\norm{ \nabla \G_{\vphi_1^\varepsilon} \Phi }_{L^2(\Omega)}^\frac23
\norm{ \nabla \G_{\vphi_1^\varepsilon} \Phi }_{H^1(\Omega)}^\frac43\\
&\quad
+ C  \norm{ D^2 \G_{\vphi_1^\varepsilon} \Phi  }_{L^4(\Omega)} 
\norm{ \nabla \G_{\vphi_1^\varepsilon} \Phi }_{L^2(\Omega)}^\frac12
\norm{ \nabla \G_{\vphi_1^\varepsilon} \Phi }_{H^1(\Omega)}^\frac12.
\end{align*}
Thanks to \eqref{eps-2} and \eqref{G-H1}, it follows that
\begin{align*}
\left\| \nabla \left( b'(\vphi_1^\varepsilon) |\nabla \G_{\vphi_1^\varepsilon} \Phi|^2\right) \right\|_{L^2(\Omega)}
&\leq 
C
\norm{ \nabla \vphi_1^\varepsilon }_{H^1(\Omega)}^\frac23 
\norm{\nabla \G_{\vphi_1^\varepsilon} \Phi }_{H^1(\Omega)}^\frac43
+ C  \norm{ D^2 \G_{\vphi_1^\varepsilon} \Phi  }_{L^4(\Omega)} 
\norm{ \nabla \G_{\vphi_1^\varepsilon} \Phi }_{H^1(\Omega)}^\frac12
\\
&\leq C \left\| \nabla \vphi_1^\varepsilon \right\|_{H^1(\Omega)}^2
+C \left\| \nabla \G_{\vphi_1^\varepsilon} \Phi \right\|_{H^1(\Omega)}^2
+ C  \norm{ D^2 \G_{\vphi_1^\varepsilon} \Phi  }_{L^4(\Omega)}^\frac43,
\end{align*}
for some constant $C$ independent of $\varepsilon$.
Recalling that $\G_{\vphi_1^\varepsilon}\Phi$ is uniformly bounded in $L^4(0,T;H^2(\Omega))\cap L^\frac83(0,T;W^{2,4}(\Omega))$, and $ \varphi_1^\varepsilon$ is uniformly bounded in $L^4(0,T;H^2(\Omega))$, we deduce the desired claim \eqref{b-V}. Hence, by \eqref{conv-eps} and \eqref{b-V}, we infer that
\begin{align}
&\lim_{\varepsilon\to 0} \frac12 \int_0^T  \left( \partial_t \vphi_1^\varepsilon, b'(\vphi_1^\varepsilon) \nabla \G_{\vphi_1^\varepsilon} \Phi \cdot \nabla \G_{\vphi_1^\varepsilon} \Phi \right) \, \eta \, \d s 
\notag 
\\
&=
\lim_{\varepsilon\to 0} 
\frac12 \int_0^T  \left( \nabla \G \partial_t \vphi_1, \nabla \left( b'(\vphi_1^\varepsilon) |\nabla \G_{\vphi_1^\varepsilon} \Phi|^2\right) \right)  \, \eta \, \d s
\notag \\
&\quad +
\lim_{\varepsilon\to 0} 
\frac12 \int_0^T  \left( \nabla \G \left(\partial_t \vphi_1^\varepsilon -\partial_t \vphi_1\right), \nabla \left( b'(\vphi_1^\varepsilon) |\nabla \G_{\vphi_1^\varepsilon} \Phi|^2\right) \right)  \, \eta \, \d s 
\notag
\\
&=
 \lim_{\varepsilon\to 0} 
\frac12 \int_0^T  \left( \nabla \G \partial_t \vphi_1,  b''(\vphi_1^\varepsilon)
\nabla \vphi_1^\varepsilon  |\nabla \G_{\vphi_1^\varepsilon} \Phi|^2 
+  2 b'(\vphi_1^\varepsilon) D^2 \G_{\vphi_1^\varepsilon} \Phi \nabla \G_{\vphi_1^\varepsilon} \Phi 
\right)   \, \eta \, \d s.
\label{last-lim}
\end{align}
By the above computations, we know that $b''(\vphi_1^\varepsilon)
\nabla \vphi_1^\varepsilon  |\nabla \G_{\vphi_1^\varepsilon} \Phi|^2 
$ and $ 2 b'(\vphi_1^\varepsilon) D^2 \G_{\vphi_1^\varepsilon} \Phi \nabla \G_{\vphi_1^\varepsilon} \Phi$ are uniformly bounded in $L^2(0,T;L^2(\Omega))$ with respect to $\varepsilon$. Then, up to a subsequence, we have
\begin{equation}
\label{g1}
b''(\vphi_1^\varepsilon)
\nabla \vphi_1^\varepsilon  |\nabla \G_{\vphi_1^\varepsilon} \Phi|^2 
\rightharpoonup
g_1  \quad
\text{weakly in } L^2(0,T;L^2(\Omega)),
\end{equation}
and 
\begin{equation}
\label{g2}
2 b'(\vphi_1^\varepsilon) D^2 \G_{\vphi_1^\varepsilon} \Phi \nabla \G_{\vphi_1^\varepsilon} \Phi
\rightharpoonup g_2 \quad 
\text{weakly in } L^2(0,T;L^2(\Omega)).
\end{equation}
We are left to identify the weak limits $g_1$ and $g_2$. First, by interpolation, we observe that 
$$
\nabla \vphi_1^\varepsilon \to \nabla \vphi_1 \quad \text{strongly in } L^8(0,T;L^4(\Omega)).
$$ 
Besides, 
$\nabla \G_{\vphi_1^\varepsilon} \Phi \rightarrow \nabla \G_{\vphi_1} \Phi $ strongly in $L^2(0,T;L^2(\Omega))$ as in \eqref{G-sH1}.
Also, it follows from \eqref{conv-G-H2} that
$$
\nabla \G_{\vphi_1^\varepsilon} \Phi \rightharpoonup \nabla \G_{\vphi_1} \Phi \quad \text{weakly in } L^4(0,T;L^6(\Omega)).
$$ 
Finally, since  $b'' \in C([-1,1])$, we learn from \eqref{eps-point} that $b''(\vphi_1^\varepsilon) \to b''(\vphi_1)$ strongly in $L^8(0,T;L^{12}(\Omega))$. Thus, we infer that 
$$
b''(\vphi_1^\varepsilon)
\nabla \vphi_1^\varepsilon  |\nabla \G_{\vphi_1^\varepsilon} \Phi|^2 
\rightharpoonup b''(\vphi_1)
\nabla \vphi_1  |\nabla \G_{\vphi_1} \Phi|^2 
\quad \text{weakly in } L^1(0,T;L^1(\Omega)),
$$ 
which entails that 
$$
g_1=  b''(\vphi_1)
\nabla \vphi_1  |\nabla \G_{\vphi_1} \Phi|^2.
$$ 
In a similar matter, owing to 
\eqref{G-sH1}, \eqref{conv-G-W24}, and observing that $b'(\vphi_1^\varepsilon) \to b'(\vphi_1)$ strongly in $L^8(0,T;L^4(\Omega))$, we obtain
$$
2 b'(\vphi_1^\varepsilon) D^2 \G_{\vphi_1^\varepsilon} \Phi \nabla \G_{\vphi_1^\varepsilon} \Phi \rightharpoonup 2 b'(\vphi_1) D^2 \G_{\vphi_1} \Phi \nabla \G_{\vphi_1} \Phi
\quad \text{weakly in } L^1(0,T;L^1(\Omega)).
$$ 
This allows us to conclude that 
$$
g_2= 2 b'(\vphi_1) D^2 \G_{\vphi_1} \Phi \nabla \G_{\vphi_1} \Phi.
$$ 
Therefore, by exploiting \eqref{g1} and \eqref{g2} in \eqref{last-lim}, we deduce that 
\begin{equation}
\begin{split}
\int_0^T (\G_{\vphi_1} \partial_t \Phi, \Phi) \, \eta \, \d s
&= -\frac{1}{2} \int_0^T (\G_{\vphi_1} \Phi, \Phi) \, \partial_t \eta \, \d s\\
& \quad 
+\frac12 \int_0^T \left( \int_\Omega  \nabla \G \partial_t \vphi_1 \cdot b''(\vphi_1)\nabla \vphi_1 | \nabla \G_{\vphi_1} \Phi |^2 \, \d x \right) \, \eta \, \d s
\\
&\quad 
+ \int_0^T \left( \int_\Omega \nabla \G \partial_t \vphi_1 \cdot b'(\vphi_1) \left(D^2 \G_{\vphi_1} \Phi \nabla \G_{\vphi_1} \Phi \right) \, \d x\right) \, \eta \, \d s
\end{split}
\end{equation}
for any $\eta \in \mathcal{D}(0,T)$. The latter implies the desired conclusion \eqref{IP-time}.

\bigskip

\textbf{Estimates of the nonlinear terms.} In the rest of the proof, we denote by $C$ a generic positive constant which depends on the parameters of the system, the total mass $m$ and the initial free energy $E(\varphi_0)$.

Concerning $I_1$, we have
\begin{equation*}
\begin{split}
|I_1|
&\leq C  \norm{\partial_t \vphi_1}_{ H^{-1}_{(0)}(\Omega)} 
\norm{ b''(\vphi_1)  \nabla \vphi_1 |\nabla \G_{\vphi_1} \Phi |^2}_{L^2(\Omega)}\!.
\end{split}
\end{equation*}
Exploiting \eqref{H2} with $q=\varphi_1$ and $f=\Phi$, and owing to \eqref{inter-2}, the following inequality holds 
\begin{equation}
\label{H2-Phi}
\norm{ \mathcal{G}_{\varphi_1} \Phi }_{H^2(\Omega)}
\leq C \left( \norm{ \nabla \varphi_1 }_{L^2(\Omega)}
\norm{\varphi_1 }_{H^2(\Omega)} 
 \norm{\nabla \mathcal{G}_{\varphi_1} \Phi }_{L^2(\Omega)}+
 \norm{ \nabla \mathcal{G}_{\varphi_1} \Phi }_{L^2(\Omega)}^\frac12 
 \norm{ \nabla \Phi }_{L^2(\Omega)}^\frac12\right)\!.
\end{equation}
%
By \eqref{GN}, \eqref{H1-phi} and \eqref{H2-Phi}, we infer that
\begin{align}
|I_1| 
&\leq 
C \norm{\partial_t \vphi_1}_{H^{-1}_{(0)}(\Omega)} 
\norm{\nabla \vphi_1}_{L^6(\Omega)}
\norm{\nabla \G_{\vphi_1} \Phi}_{L^6(\Omega)}^2
\notag
\\
&\leq 
C \norm{\partial_t \vphi_1}_{H^{-1}_{(0)}(\Omega)} 
\norm{\nabla \vphi_1}_{L^2(\Omega)}^\frac13 \norm{\vphi_1}_{H^2(\Omega)}^\frac23
\norm{\nabla \G_{\vphi_1} \Phi}_{L^2(\Omega)}^\frac23 
\norm{\G_{\vphi_1} \Phi}_{H^2(\Omega)}^\frac43
\notag
\\
&\leq C \norm{\partial_t \vphi_1}_{H^{-1}_{(0)}(\Omega)} 
 \norm{\vphi_1}_{H^2(\Omega)}^\frac23
\norm{\nabla \G_{\vphi_1} \Phi}_{L^2(\Omega)}^\frac23 
\notag
\\
&\quad \times
\left( \| \vphi_1\|_{H^2(\Omega)} \norm{\nabla \G_{\vphi_1} \Phi}_{L^2(\Omega)} + \norm{\nabla \G_{\vphi_1} \Phi}_{L^2(\Omega)}^\frac12 
\| \nabla \Phi\|_{L^2(\Omega)}^\frac12 \right)^\frac43
\notag
\\
& \leq 
C \norm{\partial_t \vphi_1}_{H^{-1}_{(0)}(\Omega)} 
 \norm{\varphi_1}_{H^2(\Omega)}^2
\norm{\nabla \G_{\vphi_1} \Phi}_{L^2(\Omega)}^2
\notag
\\
&\quad 
+
C \norm{\partial_t \vphi_1}_{H^{-1}_{(0)}(\Omega)} 
 \norm{\vphi_1}_{H^2(\Omega)}^\frac23
\norm{\nabla \G_{\vphi_1} \Phi}_{L^2(\Omega)}^\frac43 
 \| \nabla \Phi\|_{L^2(\Omega)}^\frac23
 \notag
 \\
& \leq \frac18 \| \nabla \Phi\|_{L^2(\Omega)}^2 
 + C  \left( \norm{\partial_t \vphi_1}_{H^{-1}_{(0)}(\Omega)}  
 \norm{\vphi_1}_{H^2(\Omega)}^2 + \norm{\partial_t \vphi_1}_{H^{-1}_{(0)}(\Omega)}^\frac32  
 \norm{\vphi_1}_{H^2(\Omega)}\right) \norm{\nabla \G_{\vphi_1} \Phi}_{L^2(\Omega)}^2
 \notag
 \\
 \label{I1-f}
& \leq \frac18 \| \nabla \Phi\|_{L^2(\Omega)}^2 
 + C \left( \norm{\partial_t \vphi_1}_{H^{-1}_{(0)}(\Omega)}^2+ 
 \norm{\vphi_1}_{H^2(\Omega)}^4 \right)\norm{\nabla \G_{\vphi_1} \Phi}_{L^2(\Omega)}^2\!.
\end{align}
We will now deal with $I_2$. By \eqref{GN} and \eqref{H2-Phi}, we first observe that
\begin{align*}
 \norm{ \nabla \G_{\varphi_1} \Phi }_{L^4(\Omega)}  
  &\leq  C\norm{ \nabla \G_{\varphi_1} \Phi }_{L^2(\Omega)}^\frac12 
  \norm{ \G_{\varphi_1} \Phi}_{H^2(\Omega)}^\frac12
  \\
  &\leq C \left( \| \nabla \varphi_1\|_{L^2(\Omega)}^\frac12 
  \|  \varphi_1\|_{H^2(\Omega)}^\frac12
  \norm{ \nabla \mathcal{G}_{\varphi_1} \Phi}_{L^2(\Omega)}+
 \norm{ \nabla \mathcal{G}_{\varphi_1} \Phi }_{L^2(\Omega)}^\frac34 
 \norm{ \nabla \Phi }_{L^2(\Omega)}^\frac14\right)\!.
\end{align*}
Since $\varphi_1 \in L^\infty(0,\infty;H^1(\Omega))$, we obtain
\begin{equation}
\label{W14-2}
 \norm{ \nabla \G_{\vphi_1} \Phi }_{L^4(\Omega)}  
  \leq C \left(  
  \| \vphi_1\|_{H^2(\Omega)}^\frac12 
  \norm{ \nabla \mathcal{G}_{\vphi_1} \Phi }_{L^2(\Omega)}+
 \norm{ \nabla \mathcal{G}_{\vphi_1} \Phi }_{L^2(\Omega)}^\frac34 
 \norm{ \nabla \Phi }_{L^2(\Omega)}^\frac14\right)\!.
\end{equation}
On the other hand, combining \eqref{W24} with $q=\varphi_1$ and $f=\Phi$, together with \eqref{GN} and \eqref{inter-2}, we have
\begin{align*}
\norm{ \mathcal{G}_{\varphi_1} \Phi }_{W^{2,4}(\Omega)}
 &\leq C \left( 
 \norm{\nabla \varphi_1}_{L^2(\Omega)}^\frac14 
 \norm{\varphi_1}_{H^2(\Omega)}^\frac34 
 \norm{\nabla \mathcal{G}_{\varphi_1} \Phi}_{L^2(\Omega)}^\frac14  
 \norm{\mathcal{G}_{\varphi_1} \Phi}_{H^2(\Omega)}^\frac34 
 +
 \norm{ \Phi }_{L^2(\Omega)}^\frac12 
 \norm{ \nabla \Phi }_{L^2(\Omega)}^\frac12  
\right)
\\
&\leq 
 C
 \norm{\nabla \varphi_1}_{L^2(\Omega)}^\frac14 
 \norm{\varphi_1}_{H^2(\Omega)}^\frac34 
 \norm{\nabla \mathcal{G}_{\varphi_1} \Phi}_{L^2(\Omega)}^\frac14  
 \norm{\mathcal{G}_{\varphi_1} \Phi}_{H^2(\Omega)}^\frac34
 \\
&\quad 
+ C
 \| \nabla \G_{\varphi_1} \Phi\|_{L^2(\Omega)}^\frac14 
 \| \nabla \Phi\|_{L^2(\Omega)}^\frac34.
\end{align*}
Exploiting once again \eqref{H2-Phi} and the regularity  $\varphi_1 \in L^\infty(0,\infty;H^1(\Omega))$, we deduce from the above estimate that
\begin{align}
\norm{ \mathcal{G}_{\vphi_1} \Phi }_{W^{2,4}(\Omega)}
&\leq 
C   \norm{\vphi_1}_{H^2(\Omega)}^\frac32 
\norm{\nabla \mathcal{G}_{\vphi_1} \Phi}_{L^2(\Omega)}
+
C   \norm{\vphi_1}_{H^2(\Omega)}^\frac34 
\norm{\nabla \mathcal{G}_{\vphi_1} \Phi}_{L^2(\Omega)}^\frac58
\norm{\nabla \Phi}_{L^2(\Omega)}^\frac38
\notag
\\
&\quad +
 C \norm{ \nabla \mathcal{G}_{\vphi_1} \Phi }_{L^2(\Omega)}^\frac14 
 \norm{ \nabla \Phi }_{L^2(\Omega)}^\frac34.
 \label{W24-Phi}
\end{align}
Hence, observing that
\begin{equation*}
\begin{split}
|I_2|
&\leq C  \norm{\partial_t \vphi_1}_{ H^{-1}_{(0)}(\Omega)} 
\norm{ b'(\vphi_1)}_{L^\infty(\Omega)} 
\norm{ D^2 \G_{\vphi_1} \Phi}_{L^4(\Omega)} 
\norm{ \nabla \G_{\vphi_1} \Phi }_{L^4(\Omega)}\!, 
\end{split}
\end{equation*}
we infer from \eqref{W14-2} and \eqref{W24-Phi} that
\begin{align*}
|I_2 | &\leq  K_1+ \dots+ K_6,
\end{align*}
where
\begin{align*}
K_1 &= C \norm{\partial_t \vphi_1}_{ H^{-1}_{(0)}(\Omega)} \| \vphi_1\|_{H^2(\Omega)}^2 
 \norm{ \nabla \G_{\vphi_1} \Phi }_{L^2(\Omega)}^2\!,
 \\
 K_2 &= C \norm{\partial_t \varphi_1}_{ H^{-1}_{(0)}(\Omega)} \| \vphi_1\|_{H^2(\Omega)}^\frac32 
 \norm{ \nabla \G_{\vphi_1} \Phi }_{L^2(\Omega)}^\frac74
 \norm{\nabla \Phi}_{L^2(\Omega)}^\frac14\!,
 \\
 K_3 &= C \norm{\partial_t \vphi_1}_{ H^{-1}_{(0)}(\Omega)} \| \vphi_1\|_{H^2(\Omega)}^\frac54
\norm{ \nabla \G_{\vphi_1} \Phi }_{L^2(\Omega)}^\frac{13}{8}
 \norm{\nabla \Phi}_{L^2(\Omega)}^\frac38\!,
 \\
 K_4 &= C \norm{\partial_t \vphi_1}_{ H^{-1}_{(0)}(\Omega)} 
 \| \vphi_1\|_{H^2(\Omega)}^\frac34
\norm{ \nabla \G_{\vphi_1} \Phi }_{L^2(\Omega)}^\frac{11}{8}
 \norm{\nabla \Phi}_{L^2(\Omega)}^\frac58\!,
 \\
  K_5 &= C \norm{\partial_t \vphi_1}_{ H^{-1}_{(0)}(\Omega)} 
 \| \vphi_1\|_{H^2(\Omega)}^\frac12
\norm{ \nabla \G_{\vphi_1} \Phi }_{L^2(\Omega)}^\frac{5}{4}
 \norm{\nabla \Phi}_{L^2(\Omega)}^\frac34\!,
 \\
 K_6 &= C \norm{\partial_t \vphi_1}_{ H^{-1}_{(0)}(\Omega)}
\norm{ \nabla \G_{\vphi_1} \Phi }_{L^2(\Omega)}
 \norm{\nabla \Phi}_{L^2(\Omega)}\!.
\end{align*}
By the Young inequality, we have
\begin{align*}
K_1 &\leq C \left( \norm{\partial_t \vphi_1}_{ H^{-1}_{(0)}(\Omega)}^2+ \| \vphi_1\|_{H^2(\Omega)}^4
 \right) \norm{ \nabla \G_{\vphi_1} \Phi }_{L^2(\Omega)}^2 \!,
\\[5pt]
K_2 &\leq \frac{1}{40}  \norm{\nabla \Phi}_{L^2(\Omega)}^2 +
 C  \norm{\partial_t \varphi_1}_{ H^{-1}_{(0)}(\Omega)}^\frac87 \| \vphi_1\|_{H^2(\Omega)}^\frac{12}{7} \norm{ \nabla \G_{\vphi_1} \Phi }_{L^2(\Omega)}^2
 \\
 &\leq \frac{1}{40}  \norm{\nabla \Phi}_{L^2(\Omega)}^2 +
 C  \left( \norm{\partial_t \vphi_1}_{ H^{-1}_{(0)}(\Omega)}^2+ \| \vphi_1\|_{H^2(\Omega)}^4\right) \norm{ \nabla \G_{\vphi_1} \Phi }_{L^2(\Omega)}^2\!,
\\[5pt]
K_3 &\leq \frac{1}{40}  \norm{\nabla \Phi}_{L^2(\Omega)}^2 
 + C \norm{\partial_t \vphi_1}_{ H^{-1}_{(0)}(\Omega)}^\frac{16}{13} 
 \| \vphi_1\|_{H^2(\Omega)}^\frac{20}{13}
\norm{ \nabla \G_{\vphi_1} \Phi }_{L^2(\Omega)}^2
\\
&\leq \frac{1}{40}  \norm{\nabla \Phi}_{L^2(\Omega)}^2 
 + C \left( \norm{\partial_t \vphi_1}_{ H^{-1}_{(0)}(\Omega)}^2 +
 \| \vphi_1\|_{H^2(\Omega)}^4\right)
\norm{ \nabla \G_{\vphi_1} \Phi }_{L^2(\Omega)}^2\!,
 \\[5pt]
 K_4
 &\leq \frac{1}{40}  \norm{\nabla \Phi}_{L^2(\Omega)}^2 
 + C \norm{\partial_t \vphi_1}_{ H^{-1}_{(0)}(\Omega)}^\frac{16}{11}
 \| \vphi_1\|_{H^2(\Omega)}^\frac{12}{11}
\norm{ \nabla \G_{\vphi_1} \Phi }_{L^2(\Omega)}^2
\\
& \leq \frac{1}{40}  \norm{\nabla \Phi}_{L^2(\Omega)}^2 
 + C \left( \norm{\partial_t \vphi_1}_{ H^{-1}_{(0)}(\Omega)}^2 +
 \| \vphi_1\|_{H^2(\Omega)}^4 \right)
\norm{ \nabla \G_{\vphi_1} \Phi }_{L^2(\Omega)}^2\!,
 \\[5pt]
 K_5 
 &\leq \frac{1}{40}  \norm{\nabla \Phi}_{L^2(\Omega)}^2 
 + C \norm{\partial_t \vphi_1}_{ H^{-1}_{(0)}(\Omega)}^\frac{8}{5}
 \| \vphi_1\|_{H^2(\Omega)}^\frac{4}{5}
\norm{ \nabla \G_{\vphi_1} \Phi }_{L^2(\Omega)}^2
\\
&\leq \frac{1}{40}  \norm{\nabla \Phi}_{L^2(\Omega)}^2 
 + C \left( \norm{\partial_t \vphi_1}_{ H^{-1}_{(0)}(\Omega)}^2 +
 \| \vphi_1\|_{H^2(\Omega)}^4\right)
\norm{ \nabla \G_{\vphi_1} \Phi }_{L^2(\Omega)}^2\!,
 \\[5pt]
 K_6 
 & \leq \frac{1}{40}  \norm{\nabla \Phi}_{L^2(\Omega)}^2 
 + C \norm{\partial_t \vphi_1}_{ H^{-1}_{(0)}(\Omega)}^2 \norm{ \nabla \G_{\vphi_1} \Phi }_{L^2(\Omega)}^2 \!.
\end{align*}
Therefore, we conclude that 
\begin{equation}
\label{I2-f}
|I_2| \leq \frac{1}{8}  \norm{\nabla \Phi}_{L^2(\Omega)}^2 
 + C \left( \norm{\partial_t \vphi_1}_{ H^{-1}_{(0)}(\Omega)}^2 +
 \| \vphi_1\|_{H^2(\Omega)}^4\right)
\norm{ \nabla \G_{\vphi_1} \Phi }_{L^2(\Omega)}^2\!.
\end{equation}

Concerning $I_3$, we observe from \eqref{wCH1} and  \eqref{AUTO} that 
\begin{align*}
I_3
&= - \l \partial_t \vphi_2, (\G_{\vphi_1}-\G_{\vphi_2})  \Phi \r
\\ 
&= \int_{\Omega} b(\vphi_2) \nabla \mu_2 \cdot \nabla \left( (\G_{\vphi_1}-\G_{\vphi_2})  \Phi \right) \, \d x
\\
&= \int_\Omega \nabla \mu_2 \cdot \left( b(\vphi_2)-b(\vphi_1) \right)  \nabla \G_{\vphi_1}   \Phi \, \d x 
+
\int_\Omega \nabla \mu_2 \cdot b(\vphi_1) \nabla \G_{\vphi_1} \Phi  \, \d x
\\
& \quad  -
\int_\Omega \nabla \mu_2 \cdot b(\vphi_2) \nabla \G_{\vphi_2} \Phi \, \d x
\\
&= - \int_\Omega \nabla \mu_2 \cdot \left( b(\vphi_1)-b(\vphi_2) \right)  \nabla \G_{\vphi_1} \Phi \, \d x. 
\end{align*} 
Then, exploiting \eqref{GN}, \eqref{inter-2} and \eqref{W14-2} once again, we have
\begin{align}
|I_3|
&\leq  \norm{\nabla \mu_2}_{L^2(\Omega)} 
\norm{b(\vphi_1)-b(\vphi_2)}_{L^4(\Omega)} 
\norm{\nabla \mathcal{G}_{\vphi_1}\Phi}_{L^4(\Omega)}
\notag
\\
&\leq C \norm{\nabla \mu_2}_{L^2(\Omega)} 
\norm{ \Phi }_{L^4(\Omega)} 
\norm{\nabla \mathcal{G}_{\vphi_1}\Phi}_{L^4(\Omega)}
\notag
\\
&\leq 
C \norm{\nabla \mu_2}_{L^2(\Omega)} 
\norm{ \nabla \mathcal{G}_{\vphi_1} \Phi }_{L^2(\Omega)}^\frac14 
\norm{ \nabla \Phi }_{L^2(\Omega)}^\frac34 
\notag
\\
&\quad \times
\left(  \| \vphi_1\|_{H^2(\Omega)}^\frac12 
\norm{ \nabla \mathcal{G}_{\vphi_1} \Phi }_{L^2(\Omega)}+
 \norm{ \nabla \mathcal{G}_{\vphi_1} \Phi }_{L^2(\Omega)}^\frac34 
 \norm{ \nabla \Phi }_{L^2(\Omega)}^\frac14\right)
 \notag
 \\
 &\leq C \norm{\nabla \mu_2}_{L^2(\Omega)} 
  \| \vphi_1\|_{H^2(\Omega)}^\frac12 
  \norm{ \nabla \mathcal{G}_{\vphi_1} \Phi }_{L^2(\Omega)}^\frac54 
  \norm{ \nabla \Phi }_{L^2(\Omega)}^\frac34 
 \notag
 \\
 &\quad 
 + C \norm{\nabla \mu_2}_{L^2(\Omega)}  
 \norm{ \nabla \mathcal{G}_{\vphi_1} \Phi }_{L^2(\Omega)} 
 \norm{ \nabla \Phi }_{L^2(\Omega)}
 \notag
 \\
 &\leq \frac18  \norm{ \nabla \Phi }_{L^2(\Omega)}^2
 + C  \left( \norm{\nabla \mu_2}_{L^2(\Omega)}^\frac85  
 \| \vphi_1\|_{H^2(\Omega)}^\frac45 
 + \norm{\nabla \mu_2}_{L^2(\Omega)}^2  \right) 
 \norm{ \nabla \mathcal{G}_{\vphi_1} \Phi}_{L^2(\Omega)}^2
 \notag
 \\
 \label{I3-f}
  &\leq \frac18  \norm{ \nabla \Phi }_{L^2(\Omega)}^2
 + C  \left( \norm{\nabla \mu_2}_{L^2(\Omega)}^2+
 \| \vphi_1\|_{H^2(\Omega)}^4
   \right) \norm{ \nabla \mathcal{G}_{\vphi_1} \Phi }_{L^2(\Omega)}^2.
\end{align}
Also, by \eqref{inter-2}, we notice that
\begin{equation}
\label{I_4}
\Theta_0 \norm{\Phi}_{L^2(\Omega)}^2 \leq 
\frac18 \norm{\nabla \Phi}_{L^2(\Omega)}^2 
 +C
\norm{ \nabla \G_{\vphi_1} \Phi }_{L^2(\Omega)}^2 \!.
\end{equation}

We are now in the position to derive the desired continuous dependence estimate \eqref{UNIQ}. Recalling that $\norm{ \nabla \G_{\vphi_1} \Phi }_{L^2(\Omega)}$ and $\norm{\sqrt{b(\vphi_1)} \nabla \G_{\vphi_1} \Phi}_{L^2(\Omega)}$ are equivalent norms, combining the differential inequality \eqref{uniqueness estimate} with the estimates \eqref{I1-f}, \eqref{I2-f}, \eqref{I3-f} and \eqref{I_4}, 
we end up with
\begin{equation}
\ddt  \norm{\sqrt{b(\vphi_1)} \nabla \G_{\vphi_1} \Phi}_{L^2(\Omega)}^2 +  \norm{\nabla \Phi}_{L^2(\Omega)}^2 
\leq h(t) \norm{ \sqrt{b(\vphi_1)} \nabla \G_{\vphi_1} \Phi}_{L^2(\Omega)}^2
\!
\end{equation}
where 
$$
h(\cdot)= C \Big( 1+ 
\norm{\nabla \mu_2}_{L^2(\Omega)}^2 
+ \norm{\partial_t \varphi_1}_{ H^{-1}_{(0)}(\Omega)}^2 +
 \| \vphi_1\|_{H^2(\Omega)}^4 \Big) \in L^1(0,T), \quad \forall \, T>0.
$$
Therefore, a standard application of the Gronwall lemma entails that
\begin{equation}
\norm{\sqrt{b(\vphi_1)} \nabla \G_{\vphi_1}\left( \vphi_1(t)-\vphi_2(t) \right)}_{L^2(\Omega)}^2 
\leq \norm{\sqrt{b(\vphi_1)} \nabla \G_{\vphi_1} \left( \vphi_1^0-\vphi_2^0 \right)}_{L^2(\Omega)}^2 
\mathrm{e}^{\int_0^t h(s)\, \d s}, \quad \forall \, t \geq 0.
\end{equation}
Hence, if $\vphi_1^0=\vphi_2^0$, the above inequality gives the uniqueness of weak solutions. 

\section{Proof of Theorem \ref{Goal_thm} - (C) : Propagation of regularity}
\label{S-Strong}

In this section, we show that any weak solution is more regular for positive times accordingly with Theorem \ref{Goal_thm} - part (C). 
To this aim, we first prove the following result.
\begin{proposition}\label{CHstrong-solution}
Assume that $b \in C^1([-1,1])$ satisfies \eqref{m-ndeg}.
Let $\vphi_0\in H^2(\Omega)$ be such that $- \Delta \vphi_0 + \Psi'(\vphi_0) \in H^1(\Omega)$, $\overline{\varphi_0}=m \in (-1,1)$ and $\partial_\n \varphi_0=0$ on $\partial \Omega$. 
Then, there exists a unique strong solution $\varphi$ to \eqref{CH1}-\eqref{CH2} and \eqref{nCH-mu} such that, in addition to \eqref{rw1}-\eqref{rw4}, satisfies
\begin{equation}
\label{SS-phi}
\begin{split}
&\varphi \in L^\infty(0,\infty; W^{2,p}(\Omega)), \quad 
\partial_t \varphi \in L_{\uloc}^2(0 ,\infty;H^1(\Omega)),\\
&\vphi \in L^{\infty}(\Omega\times (0,\infty)) \quad\text{such that}\quad |\vphi(x,t)|<1 
\ \text{a.e. }(x,t)\in \Omega\times (0,\infty),\\
&\mu \in L^\infty(0,\infty; H^1(\Omega))\cap L_{\uloc}^4([0,\infty);H^2(\Omega)),
\quad F'(\varphi) \in L^\infty(0,\infty; L^p(\Omega)),
\end{split}
\end{equation}
for any $2 \leq p <\infty$. The strong solution fulfills the problem \eqref{CH1}-\eqref{CH2} almost everywhere in $\Omega \times (0,\infty)$.
Furthermore, we have the following estimates
\begin{align}
\label{strong est 1}
&\norm{\partial_t \vphi}_{L_{\uloc}^2([0,\infty); H^1(\Omega))} 
+ \norm{\partial_t \vphi}_{L^{\infty}(0, \infty; H^{-1}_{(0)}(\Omega))} 
\le C,
\\
\label{strong est 2}
&\norm{\mu}_{L_{\uloc}^4([0,\infty); H^2(\Omega))}+ \norm{\mu}_{L^{\infty}(0,\infty; H^1(\Omega))} 
\le C,
\\
\label{strong est 3}
&\norm{\vphi}_{L^{\infty}(0,\infty; W^{2,p}(\Omega))} +
\norm{F'(\vphi)}_{L^{\infty}(0, \infty; L^p(\Omega))} 
\le C(p),
\end{align}
and 
\begin{align}
\label{strong est 4}
&\max_{t \geq 0} \| \varphi (t)\|_{C(\overline{\Omega})} \leq 1-\delta,
\end{align}
where the constants $C$, $C(p)$ and $\delta$ only depend on the parameters of the system, the mobility $b$, the initial energy $E(\varphi_0)$, $\overline{\varphi_0}$ and $\| -\Delta \vphi_0 +F'(\vphi_0)\|_{H^1(\Omega)}$.
In addition, if $b \in C^2([-1,1])$, $\mu \in L^2_{\uloc}([0,\infty); H^3(\Omega))$.
\end{proposition}

\begin{proof}
First of all, owing to \cite[Theorem 2.2]{BB1999} (cf. the three dimensional case), there exists a unique strong solution $\vphi$ to problem \eqref{CH1}-\eqref{CH2} and \eqref{nCH-mu} on a maximal time interval $[0,T_\star)$. In particular, the solution $\vphi$ is the limit of an approximating sequence $\lbrace \vphi_\epsilon \rbrace$, where $\varphi_\epsilon$ is the solution to a suitable approximation problem associated with a regularized potential $F_\epsilon$ suitably converging to $F$ as $\epsilon \to 0$. 
Our aim is to show that $T_\star=+\infty$. More precisely, we will show that $\varphi$ satisfies \eqref{SS-phi}. To this end, we will perform global-in-time estimates leading to \eqref{strong est 1}-\eqref{strong est 3} by directly working on the solution $\vphi$. However, a rigorous proof should be carried out by making the computations below at the approximation level $\vphi_\epsilon$, where also the initial datum is properly regularized (see the proof of \cite[Theorem 4.1]{GMT2019}), and by passing ultimately to the limit. Since this procedure is nowadays well established in the literature, we will leave this step to the interested reader.
\smallskip

Differentiating \eqref{wCH2} with respect to time and taking the duality with $\partial_t \vphi$, we obtain
\begin{equation}\label{diff e test ut}
\norm{\nabla \partial_t \vphi}_{L^2(\Omega)}^2 + \int_\Omega F'' (\vphi) |\partial_t \vphi|^2 \, \d x - \Theta_0 \norm{\partial_t \vphi}_{L^2(\Omega)}^2 = \l \partial_t \mu,\partial_t \vphi\r.
\end{equation}
Using \eqref{wCH1}, we notice that
\begin{equation}\label{counts}
\l \partial_t \mu, \partial_t \vphi \r 
= -\frac{1}{2}\ddt  \int_{\Omega} b(\vphi) \left|\nabla \mu \right|^2 \, \d x  + \frac{1}{2} \int_{\Omega} b'(\vphi) \partial_t \vphi  |\nabla \mu|^2  \, \d x,
\end{equation}
almost everywhere on $(0,T_\star)$.
Collecting \eqref{diff e test ut} and \eqref{counts} yields
\begin{equation}\label{first est}
\begin{split}
\frac{1}{2} \ddt  \int_{\Omega} b(\vphi) \left|\nabla \mu \right|^2 \, \d x &+ \norm{\nabla \partial_t \vphi }_{L^2(\Omega)}^2    +\int_{\Omega} F''(\vphi ) \left|\partial_t \vphi \right|^2 \, \d x \\
&= \Theta_0 \norm{\partial_t \vphi}_{L^2(\Omega)}^2 + \frac{1}{2}  \int_{\Omega} b'(\vphi) \partial_t \vphi  |\nabla \mu|^2 \, \d x,
\end{split}
\end{equation}
almost everywhere on $(0,T_\star)$.
The way to handle the first term on the right-hand side in \eqref{first est} is well known (see \eqref{standard} below). Instead, we need to find a suitable control for the second term in the right-hand side. In order to do so we present two different approaches.
\medskip

\textbf{First approach.} We will follow line by line the proof of \cite[Corollary 2.1]{BB1999} to point out the gap in the argument. Subsequently, we will propose a way to correct it. 

We preliminary recall that $\mu= - \mathcal{G}_{\vphi} \partial_t \vphi + \overline{\Psi'(\vphi)}$ almost everywhere in  $\Omega \times (0,\infty)$, which we will repeatedly use in the rest of the proof. Exploiting first \eqref{interp}, and then  \eqref{CI} with $r=1+\tau$ for $\tau \in (0,1)$, we have
\begin{equation*}
\begin{split}
\left|  \frac{1}{2}  \int_{\Omega} b'(\vphi) \partial_t \vphi  |\nabla \mu|^2 \, \d x \right| 
&=  \left|  \frac{1}{2}  \int_{\Omega}  \partial_t \vphi \left( b'(\varphi) |\nabla \mu|^2 - \overline{b'(\varphi) |\nabla \mu|^2|} \right) \, \d x \right| 
\\
&\leq \frac12 \norm{\nabla \partial_t \vphi}_{L^2(\Omega)} 
\norm{\nabla \G \left[  b'(\varphi) \left| \nabla \mu \right|^2-\overline{b'(\vphi) \left|\nabla \mu \right|^2}\, \right]}_{L^2(\Omega)}
\\
&\leq C\sqrt{\frac{1+\tau}{\tau}} \norm{\nabla\partial_t \vphi }_{L^2(\Omega)} 
\norm{  b'(\varphi) \left| \nabla \mu \right|^2-\overline{b'(\vphi) \left|\nabla \mu \right|^2}}_{L^{1+\tau}(\Omega)}
\\
&\leq C\sqrt{\frac{1+\tau}{\tau}} \norm{\nabla\partial_t \vphi }_{L^2(\Omega)} 
\norm{\nabla \mu}_{L^{2(1+ \tau)}(\Omega)}^2\!,
\end{split}
\end{equation*}
where $C$ is a positive constant independent of $\tau$.
Therefore, exploiting Young's inequality we learn that
\begin{equation}\label{second est}
\left| \int_{\Omega} b'(\vphi)\partial_t \vphi |\nabla \mu|^2\, \d x \right| 
\le \frac{1}{2} \norm{\nabla \partial_t \vphi }_{L^2(\Omega)}^2 
+ \frac{C}{\tau} \norm{\nabla \mu }^4_{L^{2(1+\tau)}(\Omega)}, \quad 
\forall \, \tau \in (0,1),
\end{equation}
where the constant $C$ is independent of $\tau$. We highlight that the correct dependence on $\tau$ in the second term of \eqref{second est} is the main difference with \cite[Eqn. (2.48)]{BB1999}. Throughout the rest of the proof, $C$ will denote a generic positive constant independent of $\tau$ whose value may even change in the same line. In particular, $C$ will depend on the parameters of the system, the initial energy $E(\varphi_0)$ and total mass $m$.
 
Next, we proceed with the control of 
$\norm{\nabla \mu}_{L^{2(1+\tau)}(\Omega)}=\norm{\nabla \G_{\varphi} \partial_t \vphi}_{L^{2(1+ \tau)}(\Omega)}$. 
To this end, it follows from \eqref{H2-BB} with $q=\varphi$, $f= \partial_t \varphi$ and $s= 2+ \frac{2\tau}{1-\tau}=\frac{2}{1-\tau}$, for $\tau \in (0,\frac15]$, that 
\begin{equation*}
\norm{\G_{\varphi} \partial_t \varphi}_{H^2(\Omega)} 
\le  
\left(\frac{2C}{\tau(1-\tau)}\right)^{\frac{1}{2(1-\tau)}} 
\norm{\nabla \varphi}_{L^2(\Omega)}^{\frac{\tau}{1-\tau}} 
\norm{\varphi}_{H^2(\Omega)} \norm{\nabla \G_\varphi \partial_t \varphi}_{L^2(\Omega)}
+ \frac{C}{1-\tau} \norm{\partial_t \varphi}_{L^2(\Omega)}\!. 
\end{equation*}
Since $\tau \in (0,\frac15]$, there exists a positive constant $C$ such that 
\begin{equation}
\label{H2-phit}
\norm{\G_{\varphi} \partial_t \varphi}_{H^2(\Omega)} 
\le  C
\left(\frac{1}{\tau}\right)^{\frac58} 
\norm{\nabla \varphi}_{L^2(\Omega)}^{\frac{\tau}{1-\tau}} 
\norm{\varphi}_{H^2(\Omega)} \norm{\nabla \G_\varphi \partial_t \varphi}_{L^2(\Omega)}
+C \norm{\partial_t \varphi}_{L^2(\Omega)}\!. 
\end{equation}
Then, recalling that $(a+b)^p\leq a^p+ b^p$ for $a,b \in \mathbb{R}_+$ and $p\leq 1$, by exploiting \eqref{GN} with $r=2(1+\tau)$ and \eqref{H2-phit}, we find
\begin{equation*}
\begin{split}
\norm{\nabla \mu}^4_{L^{2(1+\tau)}(\Omega)}
&\leq C (1+\tau)^2 \| \nabla \mu \|_{L^2(\Omega)}^\frac{4}{1+\tau}
\norm{ \G_{\varphi} \partial_t \varphi }_{H^2(\Omega)}^\frac{4\tau}{1+\tau}
\\
&\leq 
C (1+\tau)^2 
\left( \frac{1}{\tau}\right)^{\frac{5\tau}{2(1+\tau)}}
\| \nabla \vphi\|_{L^2(\Omega)}^\frac{4\tau^2}{1-\tau^2}
\| \vphi\|_{H^2(\Omega)}^\frac{4\tau}{1+\tau}
\| \nabla \mu \|_{L^2(\Omega)}^4\\
&\quad
+ C(1+\tau)^2 
\| \nabla \mu\|_{L^2(\Omega)}^\frac{4}{1+\tau}
\| \partial_t \vphi\|_{L^2(\Omega)}^\frac{4\tau}{1+\tau}
\\
&\leq 
C(1+\tau)^2 
\left( \frac{1}{\tau}\right)^{\frac{5\tau}{2(1+\tau)}}
\| \nabla \vphi\|_{L^2(\Omega)}^\frac{4\tau^2}{1-\tau^2}
\| \vphi\|_{H^2(\Omega)}^\frac{4\tau}{1+\tau}
\| \nabla \mu \|_{L^2(\Omega)}^4\\
&\quad
+ \frac{2\tau}{1+\tau}  \| \partial_t \vphi\|_{L^2(\Omega)}^2
+ \left(\frac{1-\tau}{1+\tau}\right)
\left[C (1+\tau)^2 \right]^{\frac{1+\tau}{1-\tau}}\
\| \nabla \mu\|_{L^2(\Omega)}^\frac{4}{1-\tau}.
\end{split}
\end{equation*}
In light of $\tau \in (0,\frac15]$, and observing that the function $s \mapsto s^{-s}$ is bounded on the interval $\left( 0,\frac15 \right]$, there exists a positive constant $C$ independent of $\tau$ such that
\begin{equation*}
\begin{split}
\norm{\nabla \mu}^4_{L^{2(1+\tau)}(\Omega)}
&\leq 
C \left( \| \nabla \vphi\|_{L^2(\Omega)}^\frac{4\tau^2}{1-\tau^2}
\| \vphi\|_{H^2(\Omega)}^\frac{4\tau}{1+\tau}
\| \nabla \mu \|_{L^2(\Omega)}^4
+ \tau  \| \partial_t \vphi\|_{L^2(\Omega)}^2
+ 
\| \nabla \mu\|_{L^2(\Omega)}^\frac{4}{1-\tau}\right)\!.
\end{split}
\end{equation*}
It follows from \eqref{H1-phi} and \eqref{muDeltaphi} (cf. \eqref{H2-phi-1}) that
\begin{equation}
\label{H2-slop}
\| \vphi\|_{H^2(\Omega)}
\leq C \left( 1+ \| \nabla \mu\|_{L^2(\Omega)} \right)\!,
\end{equation}
for some positive constant $C$ depending on $E(\vphi_0)$.
Then, by using \eqref{H1-phi} and \eqref{H2-slop}, we deduce that 
\begin{equation}
\label{nablamu-1}
\begin{split}
\norm{\nabla \mu}^4_{L^{2(1+\tau)}(\Omega)}
&\leq 
C \left( 
\| \nabla \mu\|_{L^2(\Omega)}^4
+\| \nabla \mu \|_{L^2(\Omega)}^{4+\frac{4\tau}{1+\tau}}
+ \| \nabla \mu\|_{L^2(\Omega)}^{4+\frac{4\tau}{1-\tau}}
+ \tau \| \partial_t \vphi\|_{L^2(\Omega)}^2
\right)\!.
\end{split}
\end{equation}
Collecting \eqref{first est}, \eqref{second est}, and \eqref{nablamu-1} all together, we arrive at 
\begin{equation}
\label{third_est}
\begin{split}
&\frac{1}{2} \ddt  \int_{\Omega} b(\vphi) \left|\nabla \mu \right|^2 \, \d x + \frac12 \norm{\nabla \partial_t \vphi }_{L^2(\Omega)}^2   +\int_{\Omega} F''(\vphi ) \left|\partial_t \vphi \right|^2 \, \d x \\
&\quad \leq C
\norm{\partial_t \vphi}_{L^2(\Omega)}^2 + 
\frac{C}{\tau} \left( \| \nabla \mu\|_{L^2(\Omega)}^4
+\| \nabla \mu \|_{L^2(\Omega)}^{4+\frac{4\tau}{1+\tau}}
+ \| \nabla \mu\|_{L^2(\Omega)}^{4+\frac{4\tau}{1-\tau}} \right)\!,
\end{split}
\end{equation}
almost everywhere in $(0,T_\star)$.
By \eqref{inter-2}, we have
\begin{equation}
\label{standard}
C \| \partial_t \vphi\|_{L^2(\Omega)}^2
\leq \frac14 \| \nabla \partial_t \vphi\|_{L^2(\Omega)}^2 +
C \| \nabla \mu\|_{L^2(\Omega)}^2.
\end{equation}
We observe that
$$
\frac{C}{\tau} \left( \| \nabla \mu\|_{L^2(\Omega)}^4
+\| \nabla \mu \|_{L^2(\Omega)}^{4+\frac{4\tau}{1+\tau}}
+ \| \nabla \mu\|_{L^2(\Omega)}^{4+\frac{4\tau}{1-\tau}} \right)
+C \| \nabla \mu\|_{L^2(\Omega)}^2 
\leq \frac{C}{\tau} \left( 1+ \| \nabla \mu\|_{L^2(\Omega)}^2 \right)^{2+\frac{2\tau}{1-\tau}}\!,
$$
for some positive constant $C$ independent of $\tau$. 
By \eqref{m-ndeg} and $F'' (\cdot)\geq 0$, we end up with
\begin{equation}
\frac12  \ddt  \int_{\Omega} b(\vphi) \left|\nabla \mu \right|^2 \, \d x + \frac14 \norm{\nabla \partial_t \vphi }_{L^2(\Omega)}^2   
\leq \frac{C}{\tau} \left( 1+ \int_\Omega b(\vphi) |\nabla \mu|^2 \, \d x \right)^{2+\frac{2\tau}{1-\tau}} \!,
\end{equation}
for any $\tau \in (0,\frac15]$, almost everywhere in $(0,T_\star)$.
Setting 
$$
B(t):= 1+ \int_{\Omega} b(\vphi(t)) \left|\nabla \mu (t)\right|^2 \, \d x,
$$
and $\rho = \frac{5\tau}{4}$,
we infer the differential inequality
\begin{equation}
\label{final-1}
 \ddt  B(t) + \frac12 \norm{\nabla \partial_t \vphi (t) }_{L^2(\Omega)}^2   
\leq \frac{K}{\rho} B^{2(1+\rho)}(t), \quad \forall \, \rho \in \left( 0, \frac14\right)\!,
\end{equation}
 almost everywhere in $(0,T_\star)$. Here, $K$ is a constant independent of $\rho$.
Notice that the above corresponds to \cite[Eqn. (2.53)]{BB1999}, up to the crucial factor $\frac{1}{\rho}$.
By \eqref{EE}, we observe that
\begin{equation}
\label{intB_1}
\int_0^T B(s) \, \d s \leq T + C, \quad \forall \, T \geq 0.
\end{equation}
Under the assumptions on $\vphi_0$ in Proposition \ref{CHstrong-solution}, it holds
\begin{equation}
\label{B0}
\begin{split}
B(0) &\le C \left( 1+ \norm{\nabla \mu(0)}_{L^2(\Omega)}^2 \right)
\leq C \left( 1+ \norm{- \Delta \vphi_0+ F'(\vphi_0)}_{H^1(\Omega)}^2 \right)
<\infty.
\end{split}
\end{equation}
Since $(B^{-2\rho})'= -2\rho B^{-(1+2\rho)} B'$, for any $\rho \in \left( 0, \frac{1}{4}\right)$, it follows from \eqref{final-1} that
\begin{equation}
\label{final-2}
-\frac{1}{2} \ddt \frac{1}{B^{2\rho}(t)} \leq  K B(t).
\end{equation}
We point out that the factor $\rho$ cancels out in \eqref{final-2}, on the contrary of \cite[Eqn. (2.56)]{BB1999}. Integrating the above inequality, we derive that  
\begin{equation}
\label{concl-1}
\begin{split}
B(t) \leq \left[ 1-2 K  B^{2 \rho}(0) \int_0^t B(s) \, \d s \right]^{-\frac{1}{2 \rho}} B(0), \quad \forall \, t \in [0,T_\ast],
\end{split}
\end{equation}
provided $T_{\ast}$ is such that
\begin{equation}
\label{concl-11}
2K B^{2 \rho}(0) \int_0^{T_{\ast}} B(s) \, \d s < 1.
\end{equation}
It is apparent that \eqref{concl-1} only provides a local bound in time on $B$, where $T_\star$ depends on the norm of the initial data. To overcome this issue, we propose an alternative argument leading to a global bound in time on $B$. We report the following Gronwall type result proved in \cite[Lemma 2.1]{liuzhang2023}:
\begin{lemma}
\label{gronwallino}
Let $f: [0,T_0] \to \mathbb{R}$ be an absolutely continuous positive function such that
\begin{equation}
\ddt f(t) \leq \frac{M}{\sigma} g(t) f^{1+\sigma}(t), \quad \text{for a.e. } t \in [0,T_0], \ \forall \, \sigma \in (0,\delta),
\end{equation} 
where $M$ is a positive constant and $g \in L^1(0,T_0)$ is a non-negative function. Then, there holds
\begin{equation}
 f(t) \leq f(0) \left( 2^{\frac{1}{\delta}} + f(0)\right)^{2^{2+16 M \int_0^{T_0} g(s)\, \d s}}\!, \quad 0\leq t \leq T_0.
\end{equation}
\end{lemma}
\noindent
Therefore, going back to \eqref{final-1}, we are in the position to apply Lemma \ref{gronwallino} setting $f=B$,  $g=B$, and choosing $M=2 K $, $\sigma=2\rho$, $\delta=\frac12$ and $T_0=T_\star$ leading to 
\begin{equation}
\label{Glob_B_1}
\sup_{ 0 \leq t \leq T_\star} B(t)\leq B(0) \left( 4 + B(0) \right) ^{2^{2+32 K  \int_0^{T_\star} B(s)\, \d s}}\!.
\end{equation}
Thanks to \eqref{Glob_B_1} and \eqref{B0}, a classical continuation argument entails that $T_\star=\infty$. In addition, in light of \eqref{intB_1} and \eqref{B0}, we conclude that
 \begin{equation}
  \label{Glob_muH1_1}
  \sup_{ 0 \leq t \leq T } \int_{\Omega} |\nabla \mu(t)|^2 \, \d x
  \leq \frac{2}{b_m} \left( C \left( 1+ \norm{-\Delta \vphi_0+ F'(\vphi_0)}_{H^1(\Omega)}^2 \right) \right)^{ {\rm exp} \left( C (1+T) \right)}\!, \quad \forall \, T>0.
 \end{equation}
Hence, the above argument provides a bound of $\|\nabla \mu (t)\|_{L^2(\Omega)}$ which grows with a double exponential rate in time. 
\medskip

\textbf{Second approach.} 
The argument presented here below will simplify the previous proof and will provide a bound of $\| \nabla \mu(t)\|_{L^2(\Omega)}$, which is independent of time. To this end, we now control the nonlinear term as follows
\begin{equation}\label{third_method}
\begin{split}
\left| \int_{\Omega} b'(\vphi )\partial_t \vphi |\nabla \mu |^2 \, \d x\right|
& \le C \norm{ \partial_t \vphi }_{L^2(\Omega)} 
\norm{\nabla \mu}_{L^4(\Omega)}^2\!. 
\end{split}
\end{equation}
Recalling \eqref{inter-2}, \eqref{H2} and \eqref{mu&der}, we observe that 
\begin{align*}
\| \mu - \overline{\mu}\|_{H^2(\Omega)} 
&= \norm{\G_{\varphi} \partial_t \varphi}_{H^2(\Omega)}
\\
&\leq 
C \left( \| \nabla \varphi\|_{L^2(\Omega)} 
\| \varphi\|_{H^2(\Omega)} \| \nabla \mathcal{G}_{\varphi} \partial_t \varphi\|_{L^2(\Omega)}+ \| \partial_t \varphi\|_{L^2(\Omega)} \right)
\\
&\leq C \left(  \| \varphi\|_{H^2(\Omega)} \| \nabla \mu\|_{L^2(\Omega)} + \| \nabla \mu\|_{L^2(\Omega)}^\frac12 \| \nabla \partial_t \varphi\|_{L^2(\Omega)}^\frac12  \right)\!.
\end{align*}
Here, we have also used \eqref{H1-phi} and \eqref{phit-mu}.
Thus, by \eqref{GN}, \eqref{inter-2} and \eqref{phit-mu}, we infer that 
\begin{align*}
\left| \int_{\Omega} b'(\vphi )\partial_t \vphi |\nabla \mu |^2 \, \d x\right|
& \leq 
C \norm{ \partial_t \vphi }_{L^2(\Omega)} 
\norm{\nabla \mu}_{L^2(\Omega)}
\norm{\mu-\overline{\mu}}_{H^2(\Omega)}
\\
& \leq
C \norm{ \nabla \G \partial_t \vphi }_{L^2(\Omega)}^\frac12  \norm{\nabla \partial_t \vphi}_{L^2(\Omega)}^\frac12 \norm{\nabla \mu}_{L^2(\Omega)} \\
& \quad \times
\left(  \| \varphi\|_{H^2(\Omega)} \| \nabla \mu\|_{L^2(\Omega)} + \| \nabla \mu\|_{L^2(\Omega)}^\frac12 \| \nabla \partial_t \varphi\|_{L^2(\Omega)}^\frac12  \right)
\\
&\leq 
C \norm{\nabla \partial_t \vphi}_{L^2(\Omega)}^\frac12 
\norm{\vphi}_{H^2(\Omega)} \norm{\nabla \mu}_{L^2(\Omega)}^\frac52
+ C  \norm{\nabla \partial_t \vphi}_{L^2(\Omega)}
 \norm{\nabla \mu}_{L^2(\Omega)}^2
 \\
 &\leq \frac14 \norm{\nabla \partial_t \vphi}_{L^2(\Omega)}^2 +
 C \| \varphi\|_{H^2(\Omega)}^\frac43 \| \nabla \mu\|_{L^2(\Omega)}^\frac{10}{3} 
 +C \| \nabla \mu\|_{L^2(\Omega)}^4
 \\
  &\leq \frac14 \norm{\nabla \partial_t \vphi}_{L^2(\Omega)}^2 +
 C\left(  \| \varphi\|_{H^2(\Omega)}^\frac43 \| \nabla \mu\|_{L^2(\Omega)}^\frac{4}{3} 
 + \| \nabla \mu\|_{L^2(\Omega)}^2 \right) \| \nabla \mu\|_{L^2(\Omega)}^2
  \\
  &\leq \frac14 \norm{\nabla \partial_t \vphi}_{L^2(\Omega)}^2 +
 C\left(  \| \varphi\|_{H^2(\Omega)}^4 + \| \nabla \mu\|_{L^2(\Omega)}^2 
 \right) \| \nabla \mu\|_{L^2(\Omega)}^2.
\end{align*}
Inserting the above inequality in \eqref{first est}, and recalling that
$$
\Theta_0 \|\partial_t \varphi \|_{L^2(\Omega)}^2 \leq \frac14 \| \nabla \partial_t \vphi\|_{L^2(\Omega)}^2 +
C \| \nabla \mu\|_{L^2(\Omega)}^2,
$$
we conclude that 
\begin{equation}
\label{final-3}
 \ddt \left( \int_{\Omega} b(\vphi) \left|\nabla \mu\right|^2 \, \d x\right) + \norm{\nabla \partial_t \vphi}_{L^2(\Omega)}^2
\leq C\left(1+ \| \varphi\|_{H^2(\Omega)}^4 + \| \nabla \mu\|_{L^2(\Omega)}^2 
 \right) \int_{\Omega} b(\vphi) \left|\nabla \mu\right|^2 \, \d x.
\end{equation}
Exploiting \eqref{mu-0inf}, \eqref{H2-phi-2} and \eqref{B0}, an application of the classical Gronwall lemma entails 
\begin{equation}
\label{Glob_muH1_2}
 \sup_{ 0 \leq t \leq 1} \int_{\Omega} |\nabla \mu(t)|^2 \, \d x
  \leq \frac{C}{b_m}  
  \left( 1+ \norm{-\Delta \vphi_0+ F'(\vphi_0)}_{H^1(\Omega)}^2 \right)\!.
 \end{equation}
In order to derive a global control independent of time, we report the 
uniform Gronwall lemma (see \cite[Chapter III, Lemma
1.1]{TEMAM})

\begin{lemma}
\label{gronwall-uni}
Let $f: [t_0, \infty) \to \mathbb{R}$ be an absolutely continuous positive function and $g,h$ two positive locally summable functions on $[t_0, \infty)$, which satisfy
\begin{equation}
\ddt f(t) \leq g(t) f(t) +h(t), \quad \text{for a.e. } t \geq t_0,
\end{equation} 
and 
$$
\int_t^{t+r} f(s) \, \d s \leq a_1, \quad 
\int_t^{t+r} g(s) \, \d s \leq a_2, \quad 
\int_t^{t+r} h(s) \, \d s \leq a_3, \quad \forall \, t \geq t_0,
$$
for some $r, a_j$ positive. 
Then, there holds
\begin{equation}
 f(t) \leq  \left( \frac{a_1}{r}+ a_3\right) \mathrm{e}^{a_2}, \quad \forall \, t \geq t_0+r.
\end{equation}
\end{lemma}
Since 
$$
\sup_{t \geq 0} \int_t^{t+1} \int_{\Omega} b(\vphi) \left|\nabla \mu\right|^2 \, \d x \, \d s\leq C_0,\quad
\sup_{t \geq 0} \int_t^{t+1} C \left(1+ \| \varphi\|_{H^2(\Omega)}^4 + \| \nabla \mu\|_{L^2(\Omega)}^2 
 \right) \, \d s \leq C_1,
$$
where $C_0$ and $C_1$ are two positive constant depending on $E(\vphi_0)$, $m$ and the parameters of the system, an application of Lemma \ref{gronwall-uni} with $t_0=0$, $r=1$ and $h=0$ implies that
\begin{equation}
\label{Glob_muH1_3}
 \sup_{t \geq 1} \int_{\Omega} |\nabla \mu(t)|^2 \, \d x
  \leq  \frac{C_0}{b_m} \, {\rm exp} (C_1).
 \end{equation}
\medskip

\textbf{Conclusions.} Let us denote by $\widetilde{C}$ a generic constant depending on the parameters of the system, the initial energy $E(\varphi_0)$, the total mass $m$, as well as $\| -\Delta \vphi_0 +F'(\vphi_0)\|_{H^1(\Omega)}$.
Thanks to \eqref{poincare} and \eqref{mediamu}, it immediately follows from \eqref{Glob_muH1_2} and \eqref{Glob_muH1_3} that 
\begin{equation}
\label{Glob_muH1_4}
\| \mu\|_{L^\infty(0,\infty; H^1(\Omega))} \leq \widetilde{C}.
\end{equation}
As a consequence, we learn from \eqref{phit-mu} that
\begin{equation}
\label{phit-inf}
 \| \partial_t \varphi \|_{L^\infty(0,\infty;  H^{-1}_{(0)}(\Omega))}\leq \widetilde{C}.
\end{equation}
Now, we integrate \eqref{final-3} on the time interval $[t,t+1]$ with $t \geq 0$. Thanks to \eqref{mu-0inf}, \eqref{H2-phi-2} and \eqref{Glob_muH1_4}, we infer that
\begin{equation}
\label{phit-H1}
\sup_{t \geq 0} \int_t^{t+1} \| \nabla \partial_t \vphi (s)\|_{L^2(\Omega)}^2 \, \d s
\leq  \widetilde{C},
\end{equation}
which gives $\partial_t \vphi \in L^2_{\uloc}([0,\infty); H^{1}_{(0)}(\Omega))$ by the conservation of mass. By \eqref{W2p}, we learn that
\begin{equation}
\label{phi2p}
\| \vphi\|_{L^\infty(0,\infty; W^{2,p}(\Omega))} 
+ \| F'(\vphi)\|_{L^\infty(0,\infty; L^p(\Omega))}
\leq \widetilde{C}(p), \quad \forall \, p \in [2,\infty).
\end{equation}
Next, we prove the separation property \eqref{strong est 4} following \cite{CG} and \cite{HW2021}. By \cite[Lemma A.6]{CG}, we find that
\begin{equation}
\label{F''p}
\| F''(\vphi)\|_{L^\infty(0,\infty; L^p(\Omega))}
\leq  \widetilde{C} \left(p\right) \!, \quad \forall \, p \in [2,\infty).
\end{equation}
By \cite[Lemma 3.2]{HW2021}, we derive that 
\begin{equation}
\label{F'W13}
\| F'(\vphi)\|_{L^\infty(0,\infty; W^{1,3}(\Omega))}
\leq   \widetilde{C} .
\end{equation}
As a consequence, we obtain that $\| F'(\vphi)\|_{L^\infty(0,\infty; L^\infty(\Omega))} \leq   \widetilde{C} $. Then, we immediately deduce from $\lim_{s\to \pm 1} |F'(s)| = \infty$ that there exists $\delta>0$ such that  \eqref{strong est 4} holds.

Finally, we derive higher order regularity on the chemical potential. Recalling that
\begin{align*}
\| \mu-\overline{\mu} \|_{H^2(\Omega)}
\leq C \left( \| \nabla \varphi\|_{L^2(\Omega)} \| \varphi\|_{H^2(\Omega)}\| \nabla \mathcal{G}_{\varphi} \partial_t \varphi\|_{L^2(\Omega)}+ \| \partial_t \varphi\|_{L^2(\Omega)} \right)\!,
\end{align*}
thanks to \eqref{Glob_muH1_4} and \eqref{phi2p}, we obtain 
\begin{align*}
\| \mu-\overline{\mu} \|_{H^2(\Omega)} 
\leq 
\widetilde{C}\left( 1+ \| \partial_t \varphi \|_{L^2(\Omega)} \right)\!.
\end{align*}
Since $L_{\uloc}^2([0,\infty);  H^{1}_{(0)}(\Omega)) \cap L^\infty(0,\infty;  H^{-1}_{(0)}(\Omega)) \hookrightarrow L^4_{\uloc}([0,\infty); L^2(\Omega))$, we infer that
\begin{equation}
\label{muH2}
\sup_{t \geq 0} \int_t^{t+1} \| \mu-\overline{\mu}\|_{H^2(\Omega)}^4 \, \d s
\leq   \widetilde{C} .
\end{equation}
Then, in light of \eqref{mediamu}, the latter entails that
$\mu \in L^4_{\uloc}([0,\infty); H^2(\Omega))$. Thus, we have proved \eqref{strong est 1}-\eqref{strong est 4}.

In addition, applying \eqref{H3}, and exploiting the further assumption $b \in C^2([-1,1])$, we have
\begin{align*}
\| \mu -\overline{\mu} \|_{H^3(\Omega)} 
&\leq C
\left( 
\norm{  \frac{b'(\varphi)}{b(\varphi)} \nabla \varphi \cdot \nabla (\mu-\overline{\mu})}_{H^1(\Omega)}
+
\norm{ \frac{\partial_t \varphi}{b(\varphi)} }_{H^1(\Omega)}
 \right)\!.
\end{align*}
We estimate the right-hand side as follows. By standard computations, we find
\begin{align*}
&\norm{  \frac{b'(\varphi)}{b(\varphi)} \nabla \varphi \cdot \nabla (\mu-\overline{\mu})}_{H^1(\Omega)}
\\
&\quad \leq
C \norm{  \frac{b'(\varphi)}{b(\varphi)} \nabla \varphi \cdot \nabla (\mu-\overline{\mu})}_{L^2(\Omega)}
+ C \norm{ \frac{b''(\varphi) b(\varphi) -b'(\varphi)^2}{b(\varphi)^2} \nabla \varphi \, (\nabla \varphi \cdot \nabla \mu)}_{L^2(\Omega)}
\\
&\qquad + C \norm{  \frac{b'(\varphi)}{b(\varphi)} D^2 \varphi \nabla \mu}_{L^2(\Omega)} 
+C \norm{ \frac{b'(\varphi)}{b(\varphi)} D^2 \mu \nabla \varphi}_{L^2(\Omega)}
\\
&\quad \leq C \| \nabla \varphi\|_{L^\infty(\Omega)} \| \nabla \mu\|_{L^2(\Omega)}
+ C \| \nabla \varphi\|_{L^\infty(\Omega)}^2 \| \nabla \mu\|_{L^2(\Omega)}
\\
&\qquad
+ C \| \varphi\|_{W^{2,4}(\Omega)} \| \nabla \mu \|_{L^4(\Omega)}
+ C \| \nabla \varphi\|_{L^\infty(\Omega)} \| \mu-\overline{\mu}\|_{H^2(\Omega)}
\end{align*}
and
\begin{align*}
\norm{ \frac{\partial_t \varphi}{b(\varphi)} }_{H^1(\Omega)}
&\leq 
\norm{ \frac{\partial_t \varphi}{b(\varphi)} }_{L^2(\Omega)}+
 \norm{  \frac{b'(\varphi)}{b^2(\varphi)} \nabla \varphi \, \partial_t \varphi}_{L^2(\Omega)}
+  \norm{ \frac{\nabla \partial_t \varphi}{b(\varphi)}}_{L^2(\Omega)}
\\
& \leq C \norm{ \partial_t \varphi}_{L^2(\Omega)} + 
C \norm{\nabla \varphi}_{L^\infty(\Omega)} \norm{\partial_t \varphi}_{L^2(\Omega)} + C \|\nabla \partial_t \varphi \|_{L^2(\Omega)}.
\end{align*}
Recalling the Sobolev embedding $W^{2,3}(\Omega) \hookrightarrow W^{1,\infty}(\Omega)$, and exploiting \eqref{phi2p}, we arrive at 
\begin{equation}
\| \mu -\overline{\mu} \|_{H^3(\Omega)}  
\leq \widetilde{C} \left( 1+ \| \mu-\overline{\mu}\|_{H^2(\Omega)} + \| \nabla \partial_t \varphi\|_{L^2(\Omega)} \right) \!.
\end{equation}
Hence, by \eqref{phit-H1} and \eqref{muH2}, we conclude that
\begin{equation}
\label{muH3}
\sup_{t \geq 0} \int_t^{t+1} \| \mu-\overline{\mu}\|_{H^3(\Omega)}^2 \, \d s
\leq  \widetilde{C}.
\end{equation}
In light of \eqref{mediamu} and \eqref{Glob_muH1_4}, the latter implies that $\mu \in L^2_{\uloc}([0,\infty); H^3(\Omega))$. Since the uniqueness of the strong solutions follows from \cite[Theorem 2.2]{BB1999} (cf. Theorem \ref{BB} - part (2) and \cite[Theorem 2.3]{S2007}), the proof of Proposition \ref{CHstrong-solution} is complete.
\end{proof}

We are now ready to demonstrate the propagation of regularity stated in Theorem \ref{Goal_thm} - part (C).

\begin{proof}[Proof of Theorem \ref{Goal_thm} - part (C)]
Let $\varphi$ be the global weak solutions originating from the initial condition $\varphi_0$. Fix $\tau>0$. Let $I= \lbrace t \in (0,1]: \partial_\n \varphi(t)=0 \text{ on } \partial \Omega \rbrace$. We recall that $I$ has full Lebesgue measure .
Since 
$$
\int_{I \cap \left(\frac{\tau}{2},\tau \right)} \left\| -\Delta \varphi +\Psi'(\varphi) \right\|_{H^1(\Omega)}^2 \, \d s \leq C,
$$ 
where $C$ is a positive constant depending on the parameters of the system,  $E(\varphi_0)$ and $m$,
there exists $t^\ast \in I \cap \left( \frac{\tau}{2}, \tau \right)$ such that 
\begin{equation}
\label{phiast}
\left\| -\Delta \varphi(t^\ast) +\Psi'(\varphi(t^\ast))  \right\|_{H^1(\Omega)}\leq \frac{2 C}{\tau} \quad \text{and} \quad \partial_\n \varphi(t^\ast)=0 \text{ on } \partial \Omega.
\end{equation}
Thanks to \eqref{phiast}, we apply Proposition \ref{CHstrong-solution} with initial datum $\varphi(t^\ast)$. By the uniqueness of weak solutions, it follows that the strong solution originating from $\varphi(t^\ast)$ coincides with $\varphi$. Thus, we deduce that 
\begin{align}
\label{s est 1}
&\norm{\partial_t \vphi}_{L_{\uloc}^2([t^\ast,\infty); H^1(\Omega))} 
+ \norm{\partial_t \vphi}_{L^{\infty}(t^\ast, \infty;  H^{-1}_{(0)}(\Omega))} 
\le C,
\\
\label{s est 2}
&\norm{\mu}_{L_{\uloc}^4([t^\ast,\infty); H^2(\Omega))}+ \norm{\mu}_{L^{\infty}(t^\ast,\infty; H^1(\Omega))} 
\le C,
\\
\label{s est 3}
&\norm{\vphi}_{L^{\infty}(t^\ast,\infty; W^{2,p}(\Omega))} +
\norm{F'(\vphi)}_{L^{\infty}(t^\ast, \infty; L^p(\Omega))} 
\le C(p),
\end{align}
and 
\begin{align}
\label{s est 4}
&\max_{t \geq t^\ast} \| \varphi (t)\|_{C(\overline{\Omega})} \leq 1-\delta,
\end{align}
where the constants $C$, $C(p)$ and $\delta$ only depend on the parameters of the system, the mobility $b$, the initial energy $E(\varphi_0)$, $\overline{\varphi_0}$ and $\tau$. The estimates \eqref{s est 1}-\eqref{s est 4} simply entail \eqref{REG-phi1}-\eqref{REG-phi2} and \eqref{separation} concluding the proof of Theorem \ref{Goal_thm} - part (C).
\end{proof}

\section{Proof of Theorem \ref{Goal_thm} - (D): Convergence to equilibrium}
\label{S-Long}

The argument herein is based on the method devised in \cite{AW2007}. For any $m\in (-1,1)$, we define the phase space 
$$
V_m = \Big\lbrace f \in H_{(m)}^1(\Omega)\cap L^\infty(\Omega): \| f\|_{L^\infty(\Omega)}\leq 1 \Big\rbrace,
$$
endowed with the metric
$
\mathrm{d}_{V_m}(f, g)=
 \| \nabla (f- g)\|_{L^2(\Omega)}.
$
By \eqref{poincare}, $V_m$ is a complete metric space.
According to Theorem \ref{Goal_thm} - (A) and (B), the problem \eqref{CH1}-\eqref{nCH-mu} generates a dynamical system on $V_m$, also called strongly continuous semigroup,
$$S(t):V_m\to V_m,\quad t\geq 0,$$
acting by the formula
$$
S(t) \varphi_0= \varphi(t), \quad \forall \, t \geq 0,
$$
where $\varphi(t) $ is the unique weak solution to \eqref{CH1}-\eqref{CH2} and \eqref{nCH-mu}. This is a one-parameter family of maps $S(t)$ on $V_m$ satisfying the properties:
\begin{enumerate}
\item[$\diamond$] $S(0)={\rm Id}_{V_m}$;
\item[$\diamond$] $S(t+\tau)=S(t)S(\tau)$, for every $t,\tau\geq 0$;
\item[$\diamond$] the function $t\mapsto S(t)\varphi_0$ is in $C([0,\infty), V_m)$, for every $\varphi_0 \in V_m$;
\item[$\diamond$] $S(t) \in C(V_m, V_m)$, for all $t\geq 0$. 
\end{enumerate}

Let $\varphi$ be the global weak solution departing from $\varphi_0$. Since $\varphi \in L^\infty(1,\infty; W^{2,p}(\Omega))$ for any $2\leq p < \infty$, and $\partial_t \varphi \in L^2(1,\infty; H_{(0)}^{-1}(\Omega))$, it follows that $\varphi \in BUC([1,\infty); W^{2-\varepsilon,p}(\Omega))$ for any $\varepsilon>0$, $2\leq p < \infty$. Then, the $\omega$-limit set 
$$
\omega(\varphi_0)= \Big\lbrace \varphi' \in W^{2-\varepsilon,p}(\Omega): \ \exists \, t_n \to \infty \text{ such that } \varphi(t_n) \to \varphi' \text{ in } W^{2-\varepsilon,p}(\Omega) \Big\rbrace,
$$
is non-empty, compact and connected in $W^{2-\varepsilon,p}(\Omega)$, for any $\varepsilon>0$ (cf. \cite[Theorem 9.1.8]{CH1998}). 
In addition, we observe that $E(\varphi)$ is a strict Lyapunov function for the dynamical system $S(t)$. Thus, we learn from \cite[Theorems 9.2.3 and 9.2.7]{CH1998} that
\begin{equation}
\label{Einf}
\exists \, E_\infty:= \lim_{t\to \infty} E(\varphi(t)) \quad \text{and} \quad E(\varphi')=E_\infty, \ \forall \, \varphi' \in \omega(\varphi_0),
\end{equation}
as well as
\begin{equation}
\omega(\varphi_0) \subseteq \mathcal{E}= \Big\lbrace \varphi' \in W^{2,p}(\Omega)\cap V_m: \ \varphi' \text{ solves } \eqref{Stat-CH1} - \eqref{Stat-CH2} \Big\rbrace,
\end{equation}
namely $\mathcal{E}$ is the set of equilibrium points associated with \eqref{CH1}-\eqref{CH2} and \eqref{nCH-mu}.

Our aim is proving the convergence of the trajectory $\varphi(t)$ to one steady state, namely $\omega(\varphi_0)$ is a singleton. To this end, let us consider $\delta$ such that $|\vphi(x,t)|\leq 1-\delta$ for all $(x,t) \in \overline{\Omega}\times [1,\infty)$ (cf. \eqref{separation}). We define $\widetilde{\Psi}: \mathbb{R} \to \mathbb{R}$ such that $\widetilde{\Psi} \in C^3(\mathbb{R})$, $\widetilde{\Psi}|_{[-1+\frac{\delta}{2},1-\frac{\delta}{2}]}=\Psi$ and $|\widetilde{\Psi}^{(j)}|$ are bounded for $j=1,2,3$. We recall that $\Psi$ is analytic in $(-1+\frac{\delta}{2},1-\frac{\delta}{2})$. 
Now, we set the functional $\widetilde{E}(\varphi): H^1_{(m)}(\Omega) \to \mathbb{R}$ by 
$$
\widetilde{E}(\varphi)= \int_{\Omega} \frac{1}{2} |\nabla \vphi|^2 +\widetilde{\Psi}(\vphi) \, \d x.
$$
It is easily seen that the Frech\'{e}t derivative $D \widetilde{E}: H^1_{(m)}(\Omega)\to H^{1}_{(0)}(\Omega)'$ of $\widetilde{E}$ is given by 
\begin{equation}
\label{Frechet}
\l D \widetilde{E}(\varphi), v \r_{H^{1}_{(0)}(\Omega)', H^1_{(0)}(\Omega)}
= \int_\Omega \nabla \varphi \cdot \nabla v + \widetilde{\Psi}'(\varphi) v \, \d x,
\quad \forall \, v \in  H^1_{(0)}(\Omega).
\end{equation} 
We report the Lojasiewicz-Simon inequality (see \cite[Proposition 6.3]{AW2007})
\begin{lemma}
\label{LS}
Let $\varphi' \in \mathcal{E}$. Then, there exist $\theta \in \left(0,\frac12\right]$, $C>0$, $\beta>0$ such that
\begin{equation}
|\widetilde{E}(\varphi)- \widetilde{E}(\varphi')|^{1-\theta}
\leq C \norm{D \widetilde{E}(\varphi)}_{H^{1}_{(0)}(\Omega)'}\!, 
\end{equation}
for all $\varphi \in H^1_{(m)}(\Omega)$ such that $\| \varphi-\varphi'\|_{H^1_{(0)}(\Omega)}\leq \beta$.
\end{lemma}

\noindent
Since $\omega(\varphi_0)$ is compact in $H^1_{(m)}(\Omega)$, we cover $\omega(\varphi_0)$ by finitely many open balls $\lbrace B_i \rbrace_{i=1}^N$ in $H^1_{(m)}(\Omega)$ centered at $\varphi'_i \in \omega(\varphi_0)$ with radius $\beta_i$, 
where $\beta_i$ is the constant from Lemma \ref{LS} corresponding to $\varphi'_i$. Recalling that $\widetilde{E}|_{\omega(\varphi_0)}=E_\infty$ and set $U:= \bigcup_{i=1}^N B_i$, there exist universal constants 
$\widetilde{\theta} \in \left(0,\frac12\right]$ and $\widetilde{C}>0$ such that
\begin{equation}
|\widetilde{E}(\varphi)- E_\infty|^{1-\widetilde{\theta}}
\leq \widetilde{C} \norm{D \widetilde{E}(\varphi)}_{H^{1}_{(0)}(\Omega)'}\!, \quad 
\forall \, \varphi \in U.
\end{equation}
Next, owing to $\mathrm{d}_{V_m}(\varphi(t), \omega(\varphi_0)) \to 0$ as $t \to \infty$, there exists $t^\star$ such that $\varphi(t) \in U$, for all $t\geq t^\star$. Let us define the function $H: [ t^\star, \infty) \to \mathbb{R}_+$ by 
$$
H(t):= \left( \widetilde{E}(\varphi(t))- E_\infty \right)^{\widetilde{\theta}}\!.
$$ 
Observing that $\widetilde{E}(\varphi(t)) =E(\varphi)(t)$ for all $t\geq 1$, and by exploiting \eqref{EE}, we compute
\begin{align*}
- \ddt H(t) 
&= - \widetilde{\theta} \left( \widetilde{E}(\varphi(t))- E_\infty \right)^{\widetilde{\theta}-1} \ddt \widetilde{E}(\varphi(t))
=- \widetilde{\theta} \left( \widetilde{E}(\varphi(t))- E_\infty \right)^{\widetilde{\theta}-1} \ddt E(\varphi(t)) 
\\
&\geq \frac{1}{\widetilde{C}} \frac{ \norm{ \sqrt{b(\varphi(t))} \nabla \mu(t)}_{L^2(\Omega)}^2 }{ \norm{D \widetilde{E}(\varphi (t))}_{H^{1}_{(0)}(\Omega)'} }
\geq \frac{b_m}{\widetilde{C} } \frac{ \norm{ \nabla \mu(t)}_{L^2(\Omega)}^2 }{ \norm{D \widetilde{E}(\varphi (t))}_{H^{1}_{(0)}(\Omega)'} }.
\end{align*}
By \eqref{wCH2} and \eqref{poincare}, we deduce that
\begin{align*}
\norm{D \widetilde{E}(\varphi (t))}_{H^{1}_{(0)}(\Omega)'}
&\leq 
\sup_{v \in H^1_{(0)}(\Omega), v \neq 0} \frac{\displaystyle \left|  
\int_\Omega -\Delta \varphi (t) v + \widetilde{\Psi}'(\varphi(t)) v \, \d x \right| }{\| \nabla v\|_{L^2(\Omega)}}
\\
& = \sup_{v \in H^1_{(0)}(\Omega), v \neq 0} \frac{ \displaystyle \left|  
\int_\Omega \left(\mu(t)-\overline{\mu(t)} \right) v \, \d x \right|}{\| \nabla v\|_{L^2(\Omega)}}
\leq C_P^2 \|\nabla \mu (t) \|_{L^2(\Omega)}.
\end{align*}
Therefore, we arrive at
$$
- \ddt H(t) \geq \frac{1}{C_P^2} \frac{b_m}{\widetilde{C} } \|\nabla \mu (t) \|_{L^2(\Omega)}, 
\quad \text{a.e. in } (t^\star, \infty).
$$
Integrating from $t^\star$ to $\infty$, we derive from \eqref{Einf} that 
$$
\int_{t^\star}^\infty \| \nabla \mu(t)\|_{L^2(\Omega)} \, \d t 
\leq \frac{C_P^2 \, \widetilde{C}}{b_m}
H(t^\star), 
$$
which gives that $\nabla \mu \in L^1(t^\star, \infty; L^2(\Omega))$. In light of \eqref{phit-mu}, we conclude that $\partial_t \varphi \in L^1(t^\star, \infty; H^{-1}_{(0)}(\Omega))$. The latter entails that $\lim_{t \to \infty} \int_{t^\star}^t \| \partial_t \varphi (\tau) \|_{H^{-1}_{(0)}(\Omega))} \, \d \tau$ exists.
Hence, observing that 
$$
 \varphi(t) = \varphi(t^\star) +  \int_{t^\star}^t \partial_t \varphi(\tau) \, \d \tau, \quad \forall \, t \geq t^\star\!,
$$
it follows that $\lim_{t \to \infty} \varphi(t)$ exists in $H^{-1}_{(0)}(\Omega)$. Setting $\varphi_\infty=\lim_{t \to \infty} \varphi(t)$, this proves the desired conclusion $\omega(\varphi_0)=\lbrace \varphi_\infty \rbrace$.

\appendix
\section{Proof of the inequality \eqref{H2-BB}}
\label{App-0}
\setcounter{equation}{0}

Consider $r=2$ in \eqref{W2r}. We find
\begin{equation*}
\norm{\G_q f}_{H^2(\Omega)}  
\leq C \left( \norm{f}_{L^2(\Omega)} + \norm{\nabla q}_{L^{\frac{2s}{s-2}}(\Omega)} \norm{\nabla \G_q f}_{L^s(\Omega)}\right)\!, \quad \forall \, s>2,
\end{equation*}
where the constant $C$ only depends on $b$.
Exploiting \eqref{GN}, we have
\begin{equation*}
\norm{\G_q f}_{H^2(\Omega)} 
\leq C \left( \norm{f}_{L^2(\Omega)}
+\sqrt{\frac{2 s}{s-2}} \norm{\nabla q}^{\frac{s-2}{s}}_{L^2(\Omega)} 
\norm{q}_{H^2(\Omega)}^{\frac{2}{s}} 
\sqrt{s} \norm{\nabla \G_q f}_{L^2(\Omega)}^{\frac{2}{s}} 
\norm{\G_q f}_{H^2(\Omega)}^{\frac{s-2}{s}}\right)\!.
\end{equation*}
An application of the Young inequality implies that
\begin{align*}
\left( 1- \frac{s-2}{s} \right) \norm{\G_q f}_{H^2(\Omega)} 
&\le C  \norm{f}_{L^2(\Omega)} 
+ \frac{2 C^\frac{s}{2} }{s}   
\left(\frac{2s^2}{s-2}\right)^{\frac{s}{4}} 
\norm{\nabla q}_{L^2(\Omega)}^\frac{s-2}{2} 
\norm{q}_{H^2(\Omega)} \norm{\nabla \G_q f}_{L^2(\Omega)}\!,
\end{align*}
which gives
\begin{equation*}
\norm{\G_q f}_{H^2(\Omega)} \le  \frac{s C}{2} \norm{f}_{L^2(\Omega)} 
+ C^\frac{s}{2} 
\left(\frac{2s^2}{s-2}\right)^{\frac{s}{4}} 
\norm{\nabla q}_{L^2(\Omega)}^{\frac{s-2}{2}} 
\norm{q}_{H^2(\Omega)} \norm{\nabla \G_q f}_{L^2(\Omega)}.
\end{equation*}
\medskip

\noindent
\textbf{Acknowledgments.} 
The authors are grateful to the referee for the helpful comments that improved  the presentation of this paper.
A. Giorgini is supported by the MUR grant Dipartimento di Eccellenza 2023-2027 of Dipartimento di Matematica, Politecnico di Milano. 
This work is supported by Gruppo Nazionale per l'Analisi Ma\-te\-ma\-ti\-ca, la Probabilit\`{a} e le loro Applicazioni (GNAMPA), Istituto Nazionale di Alta Matematica (INdAM). 

\medskip

\noindent
\textbf{Competing Interests and funding.}
The authors do not have any financial or non-financial interests that are directly or indirectly related to the work submitted for
publication.

\medskip

\noindent
\textbf{Data availability statement.}
No further data is used in this manuscript.

\end{document}